\documentclass[11pt]{amsart}
\usepackage{hyperref}
\usepackage{marvosym}
\usepackage{fullpage}
\usepackage{graphics}
\usepackage{bbm}
\usepackage{amsmath}
\usepackage{amssymb}
\usepackage{dsfont}
\usepackage{diagbox}
\usepackage{enumitem}
\usepackage{caption, subcaption}
\usepackage{tikz,xcolor}
\usetikzlibrary{intersections,patterns,arrows,decorations.pathreplacing,shapes.misc}
\usepackage{soul}
\usepackage{faktor}
\usepackage{multirow}
\usepackage{rotating}
\usetikzlibrary{shapes.misc}

 \newtheorem{theorem}{Theorem}[section]
 \newtheorem{corollary}[theorem]{Corollary}

 \newtheorem{lemma}[theorem]{Lemma}

 \newtheorem{proposition}[theorem]{Proposition}
 
 \theoremstyle{definition}
 \newtheorem{example}[theorem]{Example}
 
 \newtheorem{problem}[theorem]{Problem}
 \newtheorem{definition}[theorem]{Definition}
 
 \newtheorem{remark}[theorem]{Remark}

\definecolor{mygreen}{rgb}{0.1,0.8,0.1}
\definecolor{mygray}{rgb}{0.7,0.7,0.7}

\newcommand{\sink}{\node (0) [draw, rectangle, fill=black] at (0,0) {};}

\tikzset{MyNode/.style={draw, circle, inner sep=2pt}}

\newcommand{\tdot}[3]{\draw [fill=black,color=#3] (#1,#2) circle [radius=0.25];}

\newcommand{\beq}{\begin{equation}}
\newcommand{\eeq}{\end{equation}}

\newcommand{\beqn}{\begin{eqnarray}}
\newcommand{\eeqn}{\end{eqnarray}}

\newcommand{\luk}{Łukasiewicz}

\newcommand{\Config}[1]{\mathop{\mathrm{Config}}\left(#1\right)}
\newcommand{\Stable}[1]{\mathop{\mathrm{Stable}}\left(#1\right)}
\newcommand{\Rec}[1]{\mathop{\mathrm{Rec}}\left(#1\right)}
\newcommand{\SRec}[1]{\mathop{\mathrm{SR}}\left(#1\right)}

\newcommand{\dgr}[2]{\mathop{\mathrm{deg}^{#1}}\left(#2\right)}
\newcommand{\Stab}{\mathrm{Stab}}

\newcommand{\ccdot}{\! \cdot \!}

\newcommand{\mult}[2]{\mathop{\mathrm{mult}}\left(#1 #2\right)}

\newcommand{\Top}[1]{\mathrm{Topp}_{#1}}

\newcommand{\Z}{\mathbb{Z}}
\newcommand{\N}{\mathbb{N}}
\newcommand{\V}{\tilde{V}}
\newcommand{\1}[1]{\mathds{1}_{#1}}
\newcommand{\func}[1]{\mathrm{\mathbf{#1}}}
\newcommand{\dfunc}{\mathrm{\mathbf{deg}}}

\newcommand{\Del}[1]{\func{Del}_{#1}}

\newcommand{\inc}[1]{\mathop{\mathrm{inc}}\left(#1\right)}

\newcommand{\decSR}[1]{\mathop{\mathrm{SR}^{\mathrm{dec}}}\left(#1\right)}
\newcommand{\decRec}[1]{\mathop{\mathrm{Rec}^{\mathrm{dec}}}\left(#1\right)}

\newcommand{\PF}[1]{\mathrm{PF}_{#1}}
\newcommand{\PPF}[1]{\mathrm{PPF}_{#1}}
\newcommand{\GPF}[1]{\mathop{\mathrm{PF}}\left(#1\right)}
\newcommand{\PGPF}[1]{\mathop{\mathrm{PPF}}\left(#1\right)}
\newcommand{\incPF}[1]{\PF{#1}^{\mathrm{inc}}}
\newcommand{\incGPF}[1]{\mathop{\mathrm{PF}^{\mathrm{inc}}}\left(#1\right)}
\newcommand{\incPPF}[1]{\PPF{#1}^{\mathrm{inc}}}
\newcommand{\incPGPF}[1]{\mathop{\mathrm{PPF}^{\mathrm{inc}}}\left(#1\right)}

\newcommand{\OO}{\mathcal{O}}
\newcommand{\In}[1]{\mathop{\mathrm{in}^{\OO}}\left(#1\right)}
\newcommand{\Out}[1]{\mathop{\mathrm{out}^{\OO}}\left(#1\right)}

\newcommand{\LL}{\mathcal{L}}

\begin{document}

\title{Prime graphical parking functions and strongly recurrent configurations of the Abelian sandpile model}
\author{Thomas Selig, Haoyue Zhu}
\address{Department of Computing, School of Advanced Technology, Xi'an Jiaotong-Livepool University}
\email[corresponding author]{Thomas.Selig@xjtlu.edu.cn}
\email{Haoyue.Zhu18@student.xjtlu.edu.cn}
\keywords{Prime parking functions; Abelian sandpile model; Recurrent configurations; Bijection}

\begin{abstract}
This work investigates the duality between two discrete dynamical processes: parking functions, and the Abelian sandpile model (ASM). Specifically, we are interested in the extension of classical parking functions, called $G$-parking functions, introduced by Postnikov and Shapiro in 2004. $G$-parking functions are in bijection with recurrent configurations of the ASM on $G$. In this work, we define a notion of \emph{prime} $G$-parking functions. These are parking functions that are in a sense ``indecomposable''. Our notion extends the concept of primeness for classical parking functions, as well as the notion of prime $(p,q)$-parking functions introduced by Armon \emph{et al.} in recent work. We show that from the ASM perspective, prime $G$-parking functions correspond to certain configurations of the ASM, which we call \emph{strongly recurrent}. We study this new connection on a number of graph families, including wheel graphs, complete graphs, complete multi-partite graphs, and complete split graphs.
\end{abstract}

\maketitle


\section{Introduction and definitions}\label{sec:intro}

\subsection{Introduction}\label{subsec:intro}

This paper investigates the connection between parking functions and sandpile models, two dynamical processes in Discrete Mathematics. We begin with an informal introduction to these two processes, with more formal definitions being given in Sections~\ref{subsec:PF_intro} and \ref{subsec:ASM_intro} (these also contain detailed surveys of the literature).

Parking functions were originally introduced by Konheim and Weiss~\cite{KonWeiss} in their study of a hashing technique known as \emph{linear probing}. Informally, consider the following problems. We have $n$ cars which enter a street sequentially, and each car has a preferred parking spot in the street. What happens if a car's preferred spot is already occupied when it enters, i.e.\ if a previous car has parked there? In hashing terminology, this is known as a \emph{collision}. Konheim and Weiss's classical parking model dictates that in such a case, the car simply drives forward to the next available parking spot. A (classical) \emph{parking function} is an assignment of preferences which allows all cars to park through this process.

Since their introduction, parking functions have become a rich source of research topics in Mathematics and Theoretical Computer Science, with connections to a variety of different fields. In particular, there have been a number of parking functions considered and studied in the literature. Some of these are based on a modification of the collision rule. Others modify or generalise various characterisations of classical parking functions. A more detailed review of the existing literature around parking functions and their variations is given in Section~\ref{subsec:PF_intro}. In this paper, we are interested in so-called \emph{graphical parking functions}, or $G$-parking functions. These extend the parking process to a given graph $G$. $G$-parking functions were introduced by Postnikov and Shapiro~\cite{PostPF}, who showed a rich connection to another discrete dynamical process called the \emph{Abelian sandpile model} (ASM) on $G$. 

In the ASM, each vertex of $G$ is assigned a number of grains of sand. At each step of the process, a vertex $v$ of $G$ is selected at random, and an extra grain added to it. If this causes the number of grains at $v$ to exceed a certain threshold (usually the degree of $v$), we say that $v$ is \emph{unstable}. Unstable vertices \emph{topple}, sending grains to their neighbours in the graph. This may cause some neighbours to become unstable, and we topple these in turn. A special vertex, the \emph{sink}, allows excess grains to exit the system, and so this process eventually stabilises. Of central importance in ASM research are the so-called \emph{recurrent} configurations of the model, which appear infinitely often in its long-time running. Section~\ref{subsec:ASM_intro} gives more details on results relating to these recurrent configurations.

Notably, Postnikov and Shapiro~\cite{PostPF} exhibited a straightforward bijection between recurrent configurations for the ASM on $G$ and $G$-parking functions (see Theorem~\ref{thm:bij_rec_GPF}). It is this duality between the ASM and parking functions that we explore further in this paper. More precisely, we will be interested in a notion of \emph{primeness} for parking functions. In this context, prime parking functions are parking functions which cannot be ``decomposed'' into smaller parking functions. We extend the notion of primeness to $G$-parking functions, and connect prime $G$-parking functions to a subset of recurrent configurations for the ASM, which we refer to as \emph{strongly recurrent}.

Our paper is organised as follows. In the remainder of this section, we formally introduce parking functions and their variations, and the Abelian sandpile model. We survey some of the relevant literature relating to these two processes. In Section~\ref{sec:SR_prime}, we introduce the concepts of \emph{prime $G$-parking functions} and \emph{strongly recurrent configurations}, and show that these are in bijection with each other (Theorem~\ref{thm:bij_sr_GPPF}). In Section~\ref{sec:SR_graphs}, we study these concepts in more depth on specific families of graphs, obtaining various enumerative and bijective results. In particular, we will see that our notion of primeness extends the classical prime parking functions (see~\cite[Exercise~5.49]{StanEC}) and a more recent concept of prime $(p,q)$-parking functions introduced by Armon \emph{et al.}~\cite{ABHKMMY}. Finally, Section~\ref{sec:conc} summarises our main result and discusses possible future research directions.

Let us now fix some basic notation. We denote by $\Z$ the set of integers, and $\N$ the set of (strictly) positive integers. Throughout this paper, $n \in \N$ denotes a positive integer, and we let $[n] := \{1, \dots, n\}$. We are interested in two processes related to graphs. Here, we consider graphs $G = (V, E)$ which are undirected, finite, connected, with possible multiple edges (but no loops), and rooted at a  special vertex $s \in V$, called the \emph{sink}. We sometimes write $G = (\Gamma, s)$ for clarity ($\Gamma$ is the un-rooted graph), and will simply call $G$ a \emph{graph}. The set of non-sink vertices will be denoted $\tilde{V} := V \setminus \{s\}$. 

For two vertices $v, w \in V$, we denote $ \mult{v}{w}$ the number of edges between $v$ and $w$ (this can be $0$). For $v \in V$ and $A \subseteq V$, we denote $N^A(v)$ the multi-set of neighbours of $v$ in $A$, and $\dgr{A}{v} := \vert N^A(v) \vert = \sum\limits_{w \in A} \mult{v}{w}$. 
For simplicity, we usually just write $\dgr{}{v} = \dgr{V}{v}$ for the degree of the vertex $v$. Finally, we denote $\dfunc : V \rightarrow \N, v \mapsto \dgr{}{v}$ the degree function of $G$. Finally, for $A \subseteq V$, we write $G[A]$ for the \emph{induced} subgraph of $G$ on $A$: this is the graph with vertex set $A$, and edge set the set of edges with both end-points in $A$.

\subsection{Parking functions}\label{subsec:PF_intro}

We begin by formally introducing the notion of parking functions. Consider a one-way street with $n$ parking spots labelled $1$ to $n$, and $n$ cars, also labelled $1$ to $n$, which want to park in the street. For each $i \in [n]$, car $i$ has a preferred parking spot $p_i \in [n]$, and we say that $p = (p_1, \dots p_n)$ is a {parking preference}. Cars enter the street sequentially in order $1, \dots, n$, and first drive to find their preferred spots. If a car $i$'s preferred spot $p_i$ is unoccupied when it enters the street, then car $i$ parks there and is said to \emph{occupy} spot $i$. Otherwise, if a car $i$'s preferred spot $p_i$ has been occupied by a previous car $j < i$, then we say that a \emph{collision} occurs.

In classical parking functions, collisions are resolved as follows. If a car $i$'s preferred spot $p_i$ is already occupied when it enters the street, it drives on, and parks in the first unoccupied spot $k > p_i$ (because the street is one-way cars can only drive forwards). If no such spot exists, car $i$ directly exits the street, having failed to park. We say that $p = (p_1, \dots, p_n)$ is a (classical) \emph{parking function} if all $n$ cars are able to park through this process, and denote $\PF{n}$ the set of parking functions with $n$ cars and $n$ spots.

Figure~\ref{fig:ex_PF_valid} shows an example of this (classical) parking process for $p = (3, 1, 3, 1)$. Here, all cars are able to park, so $p \in \PF{4}$. However, if the last car's preferred spot were $3$ (or $4$), it could not find a parking place, so $(3, 1, 3, 3) \notin \PF{4}$.

\begin{figure}[ht]
 \centering
 
  \begin{tikzpicture}[scale=0.18]
    
    \node at (-3,1.2) {cars};
    \node at (-3,-4.5) {spots};       
    
    \foreach \x in {1,...,4}
	  \node at (-1+0.5+3*\x,-4.5) {$\x$};
    \draw [thick, color=blue] (1,-2)--(1,-3)--(4,-3)--(4,-2);
    \draw [thick, color=blue] (4,-2)--(4,-3)--(7,-3)--(7,-2);
    \draw [thick, color=blue] (7,-2)--(7,-3)--(10,-3)--(10,-2);
    \draw [thick, color=blue] (10,-2)--(10,-3)--(13,-3)--(13,-2);
    \node [MyNode, color=red, scale=0.8] (1) at (8.5,1.2) {$1$};
    \node at (8.5, -1.8) {$\downarrow$};
    \node at (15.5, -0.3) {$\Longrightarrow$};

    \begin{scope}[shift={(17.5,0)}]
    \foreach \x in {1,...,4}
	  \node at (-1+0.5+3*\x,-4.5) {$\x$};
    \draw [thick, color=blue] (1,-2)--(1,-3)--(4,-3)--(4,-2);
    \draw [thick, color=blue] (4,-2)--(4,-3)--(7,-3)--(7,-2);
    \draw [thick, color=blue] (7,-2)--(7,-3)--(10,-3)--(10,-2);
    \draw [thick, color=blue] (10,-2)--(10,-3)--(13,-3)--(13,-2);
    \node [MyNode, scale=0.8pt] (1) at (8.5,-1) {$1$};
    \node [MyNode, color=red, scale=0.8pt] (2) at (2.5,1.2) {$2$};
    \node at (2.5, -1.8) {$\downarrow$};
    \node at (15.5, -0.3) {$\Longrightarrow$};
    \end{scope}

    \begin{scope}[shift={(35,0)}]
    \foreach \x in {1,...,4}
	  \node at (-1+0.5+3*\x,-4.5) {$\x$};
    \draw [thick, color=blue] (1,-2)--(1,-3)--(4,-3)--(4,-2);
    \draw [thick, color=blue] (4,-2)--(4,-3)--(7,-3)--(7,-2);
    \draw [thick, color=blue] (7,-2)--(7,-3)--(10,-3)--(10,-2);
    \draw [thick, color=blue] (10,-2)--(10,-3)--(13,-3)--(13,-2);
    \node [MyNode, scale=0.8pt] (1) at (8.5,-1) {$1$};
    \node [MyNode,scale=0.8pt] (2) at (2.5,-1) {$2$};
    \node [MyNode, color=red, scale=0.8pt] (3) at (8.5,2.2) {$3$};
    \node at (11.5, 2.2) {$\curvearrowright$};
    \node at (15.5, -0.3) {$\Longrightarrow$};
    \end{scope}

    \begin{scope}[shift={(52.5,0)}]
    \foreach \x in {1,...,4}
	  \node at (-1+0.5+3*\x,-4.5) {$\x$};
    \draw [thick, color=blue] (1,-2)--(1,-3)--(4,-3)--(4,-2);
    \draw [thick, color=blue] (4,-2)--(4,-3)--(7,-3)--(7,-2);
    \draw [thick, color=blue] (7,-2)--(7,-3)--(10,-3)--(10,-2);
    \draw [thick, color=blue] (10,-2)--(10,-3)--(13,-3)--(13,-2);
    \node [MyNode, scale=0.8pt] (1) at (8.5,-1) {$1$};
    \node [MyNode,scale=0.8pt] (2) at (2.5,-1) {$2$};
    \node [MyNode,scale=0.8pt] (2) at (11.5,-1) {$3$};
    \node [MyNode, color=red, scale=0.8pt] (3) at (2.5,2.2) {$4$};
    \node at (5.5, 2.2) {$\curvearrowright$};
    \node at (15.5, -0.3) {$\Longrightarrow$};
    \end{scope}

    \begin{scope}[shift={(70,0)}]
    \foreach \x in {1,...,4}
	  \node at (-1+0.5+3*\x,-4.5) {$\x$};
    \draw [thick, color=blue] (1,-2)--(1,-3)--(4,-3)--(4,-2);
    \draw [thick, color=blue] (4,-2)--(4,-3)--(7,-3)--(7,-2);
    \draw [thick, color=blue] (7,-2)--(7,-3)--(10,-3)--(10,-2);
    \draw [thick, color=blue] (10,-2)--(10,-3)--(13,-3)--(13,-2);
    \node [MyNode, scale=0.8pt] (1) at (8.5,-1) {$1$};
    \node [MyNode,scale=0.8pt] (2) at (2.5,-1) {$2$};
    \node [MyNode,scale=0.8pt] (2) at (11.5,-1) {$3$};
    \node [MyNode, scale=0.8pt] (4) at (5.5,-1) {$4$};
    \end{scope}
    
  \end{tikzpicture}
    \caption{Illustrating the parking process for the parking function $p = (3, 1, 3, 1)$.} \label{fig:ex_PF_valid}
\end{figure}
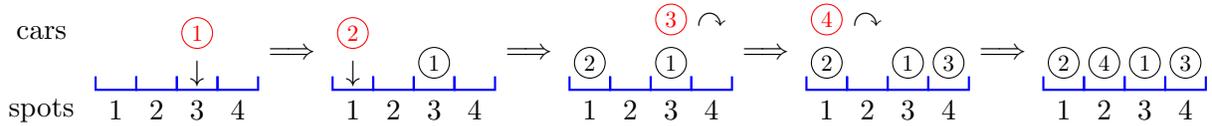

Parking functions were originally introduced by Konheim and Weiss~\cite{KonWeiss} to study hashing functions.
Since their introduction, they have been a rich research topic in Mathematics and Computer Science, with connections to various fields such as graph theory, hyperplane arrangements, discrete geometry, and the Abelian sandpile model~\cite{ChevPyl, CorPou, CR, StanHyp, PitStan}. The interested readers are referred to the excellent survey by Yan~\cite{Yan}. For our purposes, we recall the following well-known characterisations of parking functions.

\begin{theorem}\label{thm:PF_charac}
Let $p = (p_1, \dots, p_n) \in [n]^n$ be a parking preference. Then the following are equivalent.
\begin{enumerate}
\item We have $p \in \PF{n}$.
\item For any $i \in [n]$, we have $\inc{p}_i \leq i$, where $\inc{p} := \left( \inc{p}_1, \dots, \inc{p}_n \right)$ denotes the non-decreasing re-arrangement of $p$.
\item For any $i \in [n]$, we have $\left\vert \{ j \in [n]; \, p_j \leq i \} \right\vert \geq i$.
\item For any subset $S \subseteq [n]$, there exists $i \in S$ such that $p_i \leq n+1-\vert S \vert$.
\end{enumerate}
\end{theorem}

\begin{proof}
Conditions~(2) and (3) can be found for example in~\cite[Section~1.1]{Yan}. Condition~(4) is perhaps less commonly seen in the literature, but it is convenient to state it here to draw a direct comparison with the Abelian sandpile model, and particularly Dhar's burning algorithm (see Section~\ref{subsec:ASM_intro}) and $G$-parking functions (Definition~\ref{def:GPF}). Its equivalence to Condition~(3) is straightforward. Indeed, if $S$ is such that $p_i > n+1-\vert S \vert$ for all $i \in S$, then 
$ \left\vert \{ j \in [n]; \, p_j \leq n+1-\vert S \vert \} \right\vert \leq \left\vert [n] \setminus S \right\vert < n + 1 - \vert S \vert$, breaking Condition~(3). Conversely, if $i$ is such that $\left\vert \{ j \in [n]; \, p_j \leq i \} \right\vert < i$, then by construction $S := \{ j \in [n]; \, p_j > i \}$ breaks Condition~(4), since we have $\vert S \vert > n - i$, i.e.\ $i \geq n+1-\vert S \vert$, so that for all $j \in S$, we have $p_j > i \geq n+1-\vert S \vert $.
\end{proof} 

In recent years, a fruitful research direction has focused on defining and studying variations on the classical parking functions introduced above. These variations fall into a few different categories.

\begin{itemize}
\item There are parking models which consider the classical process but with $m$ cars and $n$ spots. These are called $\mathbf{(m,n)}$\textbf{-parking functions} when $m \leq n$ (see~\cite{KY_PF} for a lovely bijection to spanning forests), or \textbf{defective parking functions} when $m < n$, since in this case at least $n-m$ cars will always be unable to park (see~\cite{CamDefPF} for elegant enumerative results).
\item There are parking models with cars of different sizes, called \textbf{parking assortments} or \textbf{parking sequences} depending on the precise collision rule~\cite{EhrHapp, HarrisSeq1, HarrisSeq2}.
\item There are models which modify the collision rule, such as \textbf{Naples parking functions}~\cite{HarrisNaples1, HarrisNaples2} or \textbf{pullback parking functions}~\cite{HarrisPullback} which allow some backwards movement, \textbf{MVP parking functions}~\cite{HarrisMVP, SeligMVP} where priority is given to later cars in the sequence, \textbf{friendship parking functions}~\cite{SeligFriend} which impose additional conditions on where cars are allowed to park, and so on.
\item There are models which allow cars to have a preferred subset of spots to park in, rather than a single preference. These are called \textbf{subset parking functions}~\cite{Spiro}. See also~\cite{CDMY} for an in-depth treatment of \textbf{interval parking functions}, a special case where each preference subset is a sub-interval of $[n]$, and~\cite{AEGHH} for generalisations of these.
\item Finally, there are models that are defined by modifying some of the characterisations of parking functions from Theorem~\ref{thm:PF_charac}. For example, \textbf{vector parking functions} (or $\mathbf{u}$-parking functions)~\cite{Yan2023, Yin2023} modify Condition~(2) by replacing ``$i$'' in the right-hand side of the inequality with a vector coordinate $u_i$. \textbf{Rational parking functions}~\cite{AvBer2016, GMV2016, Nels2025} correspond to the special case where $\mathbf{u}$ is affine. The multidimensional $\mathbf{U}$-parking functions~\cite{YanUpq, Yan2023}, where $\mathbf{U}$ is a collection of multidimensional vectors, further generalise these to higher dimensions.
\end{itemize}

Another variant in a similar spirit to vector parking functions are the \emph{graphical parking functions}, or $G$-parking functions, introduced by Postnikov and Shapiro~\cite{PostPF}. These generalise Condition~(4) of Theorem~\ref{thm:PF_charac} as follows.

\begin{definition}\label{def:GPF}
Let $G = (\Gamma, s)$ be a graph, with vertex set (resp.\ non-sink vertex set) $V$ (resp.\ $\V$). For $S \subseteq V$ (here, $V = \tilde{V} \cup \{ s \}$ includes the sink), we denote $S^c := V \setminus S$ its complement in $V$. We say that a funciton $p : \V \rightarrow \N$ is a \emph{$G$-parking function} if, for every non-empty subset $S \subseteq \V$, there exists $v \in S$ such that $p(v) \leq \dgr{S^c}{v}$. We denote by $\GPF{G}$ the set of  $G$-parking functions.
\end{definition}

\begin{remark}\label{rem:GPF_0_1}
Usually, $G$-parking functions are defined as functions which take non-negative values (rather than strictly positive), with a strict inequality $p(v) < \dgr{S^c}{v}$ used. The definitions are obviously equivalent. Our preference for the above one is driven by the classical definition for parking functions taking values in $[n] = \{1, \dots, n\}$ instead of $\{0, \dots, n-1\}$ (see Proposition~\ref{pro:GPF_PF_comp}), and to avoid carrying a ``-1'' decrement in the bijection from $G$-parking functions to recurrent configurations of the ASM (see Theorem~\ref{thm:bij_rec_GPF}).
\end{remark}

\begin{remark}\label{rem:GPF_not_empty}
The set $\GPF{G}$ is always non-empty, since the constant function $\mathbf{1}: \V \rightarrow \N, v \mapsto 1$ is always a $G$-parking functions. Indeed, first note that for any $S \subseteq \V$, the complement $S^c$ is non-empty (it contains the sink $s$). Now, since $G$ is connected, for any subset $S \subseteq \V$, there must exist an edge from $S$ to $S^c$, say with end-points $v \in S$ and $w \in S^c$. In particular, we have $\dgr{S^c}{v} \geq 1$, and so $\mathbf{1}$ is a $G$-parking function, as desired.
\end{remark}

\begin{remark}\label{rem:GPF_upper_bound}
Another straightforward observation is that, if $p$ is a $G$-parking function, then for all $v \in \V$, we must have $p(v) \leq \dgr{}{v}$. We obtain this by simply taking $S = \{v\}$ in this case.
\end{remark}

We end this introductory section by explaining how this notion generalises that of classical parking functions.

\begin{proposition}\label{pro:GPF_PF_comp}
For $n \in \N$, let $K_n^0$ denote the complete graph with vertex set $[n] \cup \{0\}$, rooted at $0$. Then we have $\PF{n} = \GPF{K_n^0}$.
\end{proposition}

\begin{proof}
In the graph $K_n^0$, for any subset $S \subseteq [n]$ and vertex $v \in S$, we have $\dgr{S^c}{v} = \left\vert S^c \right\vert = \left\vert [n] \cup \{0\} \right\vert - \vert S \vert = n + 1 - \vert S \vert$, and the result follows immediately from Definition~\ref{def:GPF} and Theorem~\ref{thm:PF_charac}, Condition~(4).
\end{proof}

\subsection{The Abelian sandpile model}\label{subsec:ASM_intro}

In this section, we formally define the Abelian sandpile model, which is the second process we are interested in in this paper. The sandpile model was initially introduced by Bak, Tang and Wiesenfeld~\cite{BTW1,BTW2} as the model to study a phenomenon called \emph{self-organised criticality}. Shortly after, Dhar~\cite{Dhar1} formalised and generalised it, naming it the \emph{Abelian sandpile model} (ASM).

Let $G = (\Gamma, s)$ be a graph with vertex set (resp.\ non-sink vertex set) $V$ (resp.\ $\V$). A (sandpile) \emph{configuration} is a function $c : \V \rightarrow \Z$, where we think of $c(v)$ as representing the number of ``grains of sand'' at the vertex $v$. We define $\Config{G}$ the set of all sandpile configurations on $G$.  For $v \in \V$ and $c \in \Config{G}$, a vertex $v$ is said to be \emph{stable} in the configuration $c$ if $c(v) < \dgr{}{v}$. The configuration $c$ is \emph{stable} if all its vertices are stable, and we denote $\Stable{G}$ the set of stable configurations on $G$. Finally, for $c \in \Config{G}$, we define $\vert c \vert := \sum\limits_{v \in \V} c(v)$ to be the total number of grains in the configuration $c$.

Unstable vertices \emph{topple} by sending one grain along each of their incident edges. Formally, if $v$ is an unstable vertex in a configuration $c$, we define the \emph{toppling operator at $v$}, denoted $\Top{v}$, by:
\begin{equation}\label{eq:def_topple}
\Top{v}(c) := c - \dgr{}{v} \ccdot \1{v} + \sum\limits_{w \in \V} \mult{v}{w} \ccdot \1{w},
\end{equation}
where $\1{w}$ denotes the indicator function at $w$ (i.e.\ $\1{w}(u) = 1$ if $u = w$ and $0$ otherwise). In words, $\Top{v}(c)$ is obtained from the configuration $c$ by removing $\dgr{}{v}$ grains at $v$, and re-distributing them to the neighbours of $v$ according to the multiplicities of the edges $vw$.

Toppling an unstable vertex $v$ may then cause some of its neighbours to become unstable, and we topple these in turn. Now note that by applying Equation~\eqref{eq:def_topple}, we have $\vert \Top{v}(c) \vert = \vert c \vert - \mult{v}{s}$. As such, the total number of grains in a configuration strictly decreases when we topple a neighbour of the sink, and otherwise remains the same. We think of the sink as ``absorbing'' grains when we topple one of its neighbours, with these grains exiting the system.

From this, it is relatively straightforward to see that, starting from an unstable configuration $c$ and toppling unstable vertices successively, we will eventually reach a stable configuration $c'$. Furthermore, Dhar~\cite{Dhar1} showed that the stable configuration $c'$ reached through this process does not depend on the order in which vertices are toppled.
We write $c' = \Stab(c)$ and call it the \emph{stabilisation} of $c$. Figure~\ref{fig:stab} shows the stabilisation process for the configuration $c = \left( c(v_1), c(v_2), c(v_3), c(v_4) \right) = (4, 1, 1, 0)$. In this figure, \textcolor{red}{red} denotes unstable vertices, \textcolor{blue}{blue} edges represent grains being sent along those edges through a toppling.

\begin{figure}[ht]
  \centering
  \begin{tikzpicture}[scale=0.55]
  
 \begin{scope}
  \node [MyNode, color=red] (1) at (-2,2) {$4$};
  \node [MyNode] (2) at (-0.5,4) {$1$};  
  \node [MyNode] (3) at (1,2) {$1$};
  \node [MyNode] (4) at (2.5,1.5) {$0$};
  \sink
  \foreach \xx in {2,3,4}
    \node [above, yshift=6pt] at (\xx) {$v_{\xx}$};
  \node [above, color=red, yshift=6pt] at (1) {$v_1$};
  \node [below, yshift=-4pt] at (0) {$s$};
  \draw [thick, blue] (0)--(1);
  \draw [thick, out=180, in=-90, blue] (0) to (1); 
  \draw [thick, blue] (1)--(2);
  \draw [thick, blue] (1)--(3);
  \draw [thick] (0)--(3);
  \draw [thick] (2)--(3);
  \draw [thick] (0)--(4);
  \draw [->] (3.8,2)--(4.8,2);
  \node at (4.3,1.5) {\textcolor{red}{$v_1$}};  
 \end{scope}  
  
 \begin{scope}[shift={(8,0)}]
  \node [MyNode] (1) at (-2,2) {$0$};
  \node [MyNode, color=red, fill=white] (2) at (-0.5,4) {$2$};  
  \node [MyNode] (3) at (1,2) {$2$};
  \node [MyNode] (4) at (2.5,1.5) {$0$};
  \sink
  \foreach \xx in {1,3,4}
    \node [above, yshift=6pt] at (\xx) {$v_{\xx}$};
  \node [above, color=red, yshift=6pt] at (2) {$v_2$};
  \node [below, yshift=-4pt] at (0) {$s$};
  \draw [thick] (0)--(1);
  \draw [thick, out=180, in=-90] (0) to (1); 
  \draw [thick, blue] (1)--(2);
  \draw [thick] (1)--(3);
  \draw [thick] (0)--(3);
  \draw [thick, blue] (2)--(3);
  \draw [thick] (0)--(4);
  \draw [->] (3.8,2)--(4.8,2);
  \node at (4.3,1.5) {\textcolor{red}{$v_2$}};  
 \end{scope}
  
 \begin{scope}[shift={(16,0)}]
  \node [MyNode] (1) at (-2,2) {$1$};
  \node [MyNode] (2) at (-0.5,4) {$0$};  
  \node [MyNode, color=red, fill=white] (3) at (1,2) {$3$};
  \node [MyNode] (4) at (2.5,1.5) {$0$};
  \sink
  \foreach \xx in {1,2,4}
    \node [above, yshift=6pt] at (\xx) {$v_{\xx}$};
  \node [above, color=red, yshift=6pt] at (3) {$v_3$};
  \node [below, yshift=-4pt] at (0) {$s$};
  \draw [thick] (0)--(1);
  \draw [thick, out=180, in=-90] (0) to (1); 
  \draw [thick] (1)--(2);
  \draw [thick, blue] (1)--(3);
  \draw [thick, blue] (0)--(3);
  \draw [thick, blue] (2)--(3);
  \draw [thick] (0)--(4);
  \draw [->] (3.8,2)--(4.8,2);
  \node at (4.3,1.5) {\textcolor{red}{$v_3$}}; 
 \end{scope}
 
  \begin{scope}[shift={(24,0)}]
  \node [MyNode] (1) at (-2,2) {$2$};
  \node [MyNode] (2) at (-0.5,4) {$1$};  
  \node [MyNode] (3) at (1,2) {$0$};
  \node [MyNode] (4) at (2.5,1.5) {$0$};
  \sink
  \foreach \xx in {1,...,4}
    \node [above, yshift=6pt] at (\xx) {$v_{\xx}$};
  \node [below, yshift=-4pt] at (0) {$s$};
  \draw [thick] (0)--(1);
  \draw [thick, out=180, in=-90] (0) to (1); 
  \draw [thick] (1)--(2);
  \draw [thick] (1)--(3);
  \draw [thick] (0)--(3);
  \draw [thick] (2)--(3);
  \draw [thick] (0)--(4);
 \end{scope}
  
  \end{tikzpicture}
  \caption{Illustrating the stabilisation for $c = \left( c(v_1),c(v_2),c(v_3),c(v_4)\right) = (4,1,1,0) \in \Config{G}$. Vertices being toppled are represented with arrows in each phase. \label{fig:stab}}
\end{figure}
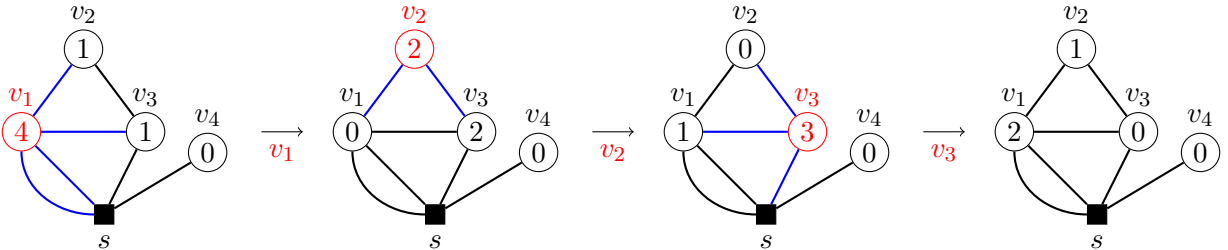

We then define a Markov chain on the set $\Stable{G}$ as follows. Let $\mu : \V \rightarrow [0,1]$ be a probability distribution on the set of non-sink vertices $\V$ such that $\mu(v)>0$ for all $v \in \V$. 
At each step of the Markov chain we add a grain at the vertex $v$ with probability $\mu(v)$ and stabilise the resulting configuration.
Formally the transition matrix $Q$ is given by:
\begin{equation}\label{eq:transition_matrix}
\forall\, c,c' \in \Stable{G}\!, \,  Q(c,c')=\sum\limits_{v \in \V} \mu(v) \ccdot \1{\Stab \left( c + \1{v} \right) \, = \, c'}.
\end{equation}

A configuration is called \emph{recurrent} if it appears infinitely often in the long-time running of this Markov chain. We denote by $\Rec{G}$ the set of recurrent configurations on $G$. These recurrent configurations are of central importance in ASM research. A seminal result shows that the number of recurrent configurations for the ASM on a graph $G$ is equal to the number of spanning trees of $G$. This can be proved algebraically using the Kirkhoff's matrix-tree theorem (see e.g.~\cite{Redig}), or bijectively (see e.g.~\cite{Ber} or \cite{CLB}). Dhar~\cite{Dhar1} showed that the set of recurrent configurations forms an Abelian group when equipped with the addition operation ``vertex-pointwise addition and stablisation''. He named this the sandpile group of the graph. 
This group, also known as the critical group, graph Jacobian, or Picard group, has been widely studied in algebraic graph and matroid theory. Its most celebrated contribution is perhaps its role in developing discrete Riemann-Roch and Abel-Jacobi theory on graphs~\cite{Baker}. 
The structure of the sandpile group has also been explicitly computed on a variety of graph families, including complete graphs~\cite{CR}, nearly-complete graphs (complete graphs with a cycle removed)~\cite{Zhou}, wheel graphs~\cite{Biggs1}, Cayley graphs of the dihedral group $D_n$~\cite{DFF}, and many more (see e.g.~\cite{Chen} and the references therein for a more complete list).

In enumerative and bijective combinatorics, a fruitful direction of recent ASM research has also focused on graph families, but instead of looking at the sandpile group, the focus has been on calculating the set of recurrent configurations via bijections to other more easily enumerated or generated combinatorial objects. The graph families that have been studied include complete graphs~\cite{CR} (see also~\cite{SelSSM}), complete bipartite or multi-partite graphs with a dominating sink~\cite{CorPou}, complete bipartite graphs where the sink is in one of the two components~\cite{DLB, SeligCompBipart} (see also~\cite{AD, AADB, ALB}), wheel and fan graphs~\cite{SelWheel}, complete split graphs~\cite{DDLB, Duk}, Ferrers graphs~\cite{DSSS1, SSS}, permutation graphs~\cite{DSSS2},  and so on. We will make use of some of these studies in Section~\ref{sec:SR_graphs} in this paper to characterise prime parking functions (or equivalently strongly recurrent configurations of the ASM) for some of these graph families.

We now give three useful characterisations of recurrent configurations of the ASM. The first (Condition~(2)) is in terms of so-called forbidden subconfigurations. For a configuration $c$ on a graph $G$, we say that $F \subset \V$ is a \emph{forbidden subconfiguration} for $c$ if $F \neq \emptyset$ and, for all $v \in F$, we have $c(v) < \dgr{F}{v}$. In words, this means that the configuration $c$ restricted to the induced subgraph $G[F]$ is stable. The second and third characterisations (Conditions~(3) and (4)) are popular equivalent framings of Dhar's \emph{burning criterion}. The three characterisations are part of the ``folklore'' of ASM research, see e.g.~\cite[Section~6]{Dhar} or \cite[Section~3]{Redig}.

\begin{theorem}\label{thm:charac_rec}
Let $G = (\Gamma, s)$ be a graph with vertex set (resp.\ non-sink vertex set) $V$ (resp.\ $\V$), and $c \in \Stable{G}$ a \emph{stable} configuration on $G$. The following are equivalent.
\begin{enumerate}
\item We have $c \in \Rec{G}$.
\item There exists no forbidden subconfiguration for $c$.
\item We have $\Stab\Big(c + \sum\limits_{w \in \V} \mult{w}{s} \ccdot \1{w} \Big) = c$.
\item There exists a permutation $v_0 = s, v_1, \dots, v_{\vert \V \vert}$ of all vertices in $V$ such that, for all $i \geq 1$, we have $c(v_i) \geq \dgr{V \setminus \{v_0, \dots, v_{i-1} \}}{v_i}$.
\end{enumerate}
\end{theorem}

If $c$ is recurrent, a sequence such as in Condition~(4) is called 	a \emph{burning sequence} for $c$. 
It is sometimes convenient to think of the configuration $c + \sum\limits_{w \in \V} \mult{w}{s} \ccdot \1{w}$ in Condition~(3) as the configuration $c$ in which we have ``toppled the sink''. In fact, if $c$ is recurrent, the stabilisation of this configuration topples every vertex exactly once, following the order of any burning sequence. With this interpretation, configuration $c$ is recurrent when, starting from $c$ and toppling the sink, every vertex topples exactly once in the resulting stabilisation process, finally yielding the original configuration $c$.

\begin{remark}\label{rem:non-negative}
It is straightforward to see that recurrent configurations are \emph{non-negative}, i.e.\ $c(v) \geq 0$ for all $v \in \V$ (otherwise if $c(v) < 0$, then $\{v\}$ is a forbidden subconfiguration for $c$). It is most common to define all configurations as being non-negative: this gives a more realistic meaning to the physical interpretation of $c(v)$ representing the number of grains of sand at $v$. However, in this paper we will be interested in removing grains of sand from configurations (see Equation~\eqref{eq:def_sr}), which may cause the configuration to take negative values at some vertices, hence our approach.
\end{remark}

We end this section by recalling the following result, which establishes the key duality between $G$-parking functions and recurrent configurations for the ASM on $G$.

\begin{theorem}[{\cite[Lemma~13.6]{PostPF}}]\label{thm:bij_rec_GPF}
Let $G = (\Gamma, s)$ be a graph. Then the map $ c \mapsto \mathrm{\mathbf{deg}} - c $ is a bijection from $\Rec{G}$ to $\GPF{G}$.
\end{theorem}


\section{Prime $G$-parking functions and strong recurrence}\label{sec:SR_prime}

In this section, we introduce the concept of \emph{primeness} for $G$-parking functions. We connect this to the notion of \emph{strong recurrence} for the ASM.

\subsection{Prime classical parking functions}\label{subsec:PPF}

We begin by recalling the notion of primeness for classical parking functions. We say that a (classical) parking function $p = (p_1, \dots, p_n) \in \PF{n}$ has a \emph{breakpoint} at an index $j \in [n]$ if we have $\vert \{ i \in [n]; \, p_i \leq j \} \vert = j$. In words, this means that there are exactly $j$ cars which prefer one of the first $j$ spots in the street, and therefore exactly $n-j$ cars which prefer one of the last $n-j$ spots. For example, the parking function $p = (3, 1, 3, 1)$ from Figure~\ref{fig:ex_PF_valid} has a breakpoint at $j = 2$. 

If $j < n$, we can think of $p$ as \emph{decomposable} in the following sense. Write $s_1 < \dots < s_j$ for elements of the set $\{i \in [n]; \, p_i \leq j\}$, and $s_{j+1} < \dots < s_n$ for elements of its complement. Then by construction we have $p^{\leq j} := \left( p_{s_1}, \dots, p_{s_j} \right) \in \PF{j}$ and $p ^ {>j} := \left(p_{s_{j+1}} - j, \dots, p_{s_n} - j \right) \in \PF{n-j}$, thus decomposing $p$ into two smaller parking functions.

A \emph{prime} parking function is a parking function whose only breakpoint is at index $n$, and we denote by $\PPF{n}$ the set of prime parking functions of size $n$. Prime parking functions were introduced by Gessel (see e.g.~\cite[Exercise~5.49]{StanEC}), who showed that $\left\vert \PPF{n} \right\vert = (n-1)^{n-1}$ (see also~\cite{DG2024} for a bijective proof).

\begin{remark}\label{rem:lattice_paths}
Parking functions can be mapped to two classical families of lattice paths: Dyck paths (see e.g.~\cite[Section~1.6]{Yan}) and \luk\ paths (see~\cite{SeligLuk}). We briefly recall these two mappings. Given a parking function $p = (p_1, \dots, p_n) \in \PF{n}$, for each $j \in [n]$, we define $q_j := \left\vert \{ i \in [n]; p_i = j \} \right\vert$. We then map $p$ to the following lattice paths:
\begin{itemize}
\item The Dyck path $\mathcal{D} := U^{q_1} D U^{q_2} \dots D U^{q_n} D$, where $U = (1/2, 1/2)$ and $D = (1/2, -1/2)$;
\item The \luk\ path $\mathcal{L} := S_1 \dots S_n$, where $S_k := (1, q_k - 1)$ for $k \in [n]$.
\end{itemize}
In this setting, breakpoints of $p$ are simply the positive $x$-coordinates where the paths hit the $x$-axis, and prime parking functions are those whose corresponding paths only hit the $x$-axis at the path's end points.
In fact, if we restrict ourselves to non-decreasing parking functions (where $p_1 \leq \dots \leq p_n$), these mappings are bijective. Figure~\ref{fig:dyck_luk} shows the Dyck path (left) and \luk\ path (right) corresponding to the parking function $p = (1, 1, 1, 3, 4, 4, 7, 7, 7)$, with breakpoints (in red) at $\color{red}{j = 6}$ and $\color{red}{j = 9}$.

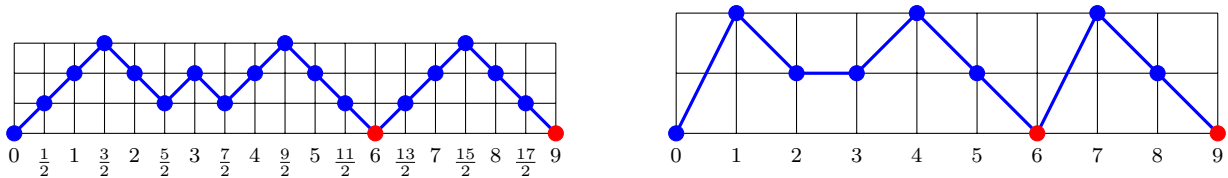
\begin{figure}[ht]
\begin{tikzpicture}[scale=0.4]
   \draw [step=1] (0,0) grid (18,3);
   \foreach \xx in {0,...,9}
     \node [below, yshift=-2pt] at (2*\xx,0) {{\scriptsize $\xx$}};
   \foreach \yy in {1,3,...,17}
     \node [below, yshift=-2pt] at (\yy,0) {{\scriptsize $\frac{\yy}{2}$}};
   \draw [very thick, color=blue]  (0,0)--(3,3)--(5,1)--(6,2)--(7,1)--(9,3)--(12,0)--(15,3)--(18,0);
   \tdot{0}{0}{blue}
   \tdot{1}{1}{blue}
   \tdot{2}{2}{blue}
   \tdot{3}{3}{blue}
   \tdot{4}{2}{blue}
   \tdot{5}{1}{blue}
   \tdot{6}{2}{blue}
   \tdot{7}{1}{blue}
   \tdot{8}{2}{blue}
   \tdot{9}{3}{blue}
   \tdot{10}{2}{blue}
   \tdot{11}{1}{blue}
   \tdot{12}{0}{red}
   \tdot{13}{1}{blue}
   \tdot{14}{2}{blue}
   \tdot{15}{3}{blue}
   \tdot{16}{2}{blue}
   \tdot{17}{1}{blue}
   \tdot{18}{0}{red}
 \begin{scope}[shift={(22,0)}]
   \draw [step=2] (0,0) grid (18,4);
   \foreach \xx in {0,...,9}
     \node [below, yshift=-2pt] at (2*\xx,0) {{\scriptsize $\xx$}};
    \draw [very thick, color=blue]  (0,0)--(2,4)--(4,2)--(6,2)--(8,4)--(10,2)--(12,0)--(14,4)--(16,2)--(18,0);
	\tdot{0}{0}{blue}
    \tdot{2}{4}{blue}
    \tdot{4}{2}{blue}
    \tdot{6}{2}{blue}
    \tdot{8}{4}{blue}
    \tdot{10}{2}{blue}
    \tdot{12}{0}{red} 
    \tdot{14}{4}{blue} 
    \tdot{16}{2}{blue}    
    \tdot{18}{0}{red}
  \end{scope}
\end{tikzpicture}
\caption{The Dyck path (left) and \luk\ path (right) corresponding to the parking function $p = (1, 1, 1, 3, 4, 4, 7, 7, 7)$. The breakpoints of $p$ are the positive $x$-coordinates at which the paths hit the $x$-axis. \label{fig:dyck_luk}}
\end{figure}
\end{remark}

\subsection{Prime $G$-parking functions}\label{subsec:PGPF}

The concept of primeness for parking functions was recently extended in~\cite{ABHKMMY}, with definitions for vector parking functions, $(p,q)$-parking functions, and two-dimensional vector parking functions. We propose here to extend the concept to $G$-parking functions. If $G = (\Gamma, s)$ is a graph with vertex set $V = \V \cup \{s\}$, and $A \subseteq \V$, we define $G^A := G[A \cup \{s\}]$ to be the induced subgraph on $A \cup \{s\}$.

\begin{definition}\label{def:prime_GPF}
Let $(A,B)$ be an ordered set partition of $\V$, with $A,B \neq \emptyset$. For $p \in \GPF{G}$, we define $p^A : A \rightarrow \N$ by $p^A(v) := p(v)$ for all $v \in A$, and $p^B : B \rightarrow \Z$ by $p^B(v) := p(v) - \dgr{A}{v}$ for all $v \in B$. We say that $p$ is \emph{decomposable} with respect to the partition $(A, B)$ if we have $p^A \in \GPF{G^A}$ and $p^B \in \GPF{G^B}$. We say that $p$ is \emph{prime} if there exists no ordered set partition $(A,B)$ with respect to which $p$ is decomposable, and denote by $\PGPF{G}$ the set of all prime $G$-parking functions.
\end{definition}

\begin{proposition}\label{pro:empty_PGPF_cut_vertex}
If $s$ is \emph{not} a cut vertex of $G$, then $\mathbf{1} \in \PGPF{G}$. Otherwise, $\PGPF{G} = \emptyset$.
\end{proposition}

\begin{proof}
Suppose first that $s$ is a cut vertex of $G$. Thus, removing $s$ from $G$ disconnects $G$ into two parts $A$ and $B$, i.e.\ with no edges between $A$ and $B$ in $G$, see Figure~\ref{fig:s_cut} (the subgraphs $G[A]$ and $G[B]$ need not necessarily be connected). Then by definition $p^B(v) = p(v)$ for all $v \in B$, since $\dgr{A}{v} = 0$. We claim that $p^A$ and $p^B$ are both parking functions. Indeed, let $S \subseteq A$. Since $p$ is a $G$-parking function, there exists $v \in S$ such that $p(v) \leq \dgr{V \setminus S}{v}$. But by construction we have $\dgr{V \setminus S}{v} = \dgr{V}{v} - \dgr{S}{v} = \dgr{A \cup \{s\}}{v} - \dgr{S}{v} + \dgr{B}{v} = \dgr{A \cup \{s\}}{v} - \dgr{S}{v} + 0 = \dgr{A \cup \{s\} \setminus S}{v}$. This shows that $p^A$ is a $G^A$-parking function, as desired, and the argument is analogous for $p^B$.

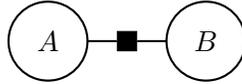
\begin{figure}[ht]
 \centering
  \begin{tikzpicture}[scale=0.35]
 \node (0) [draw, rectangle, fill=black] at (6,0) {};

 \draw [thick] (3,0) circle (1.5cm);
 \draw [thick] (9,0) circle (1.5cm);
 \draw [thick] (4.5,0)--(0)--(7.5,0);

 \node at (3,0) {$A$};
 \node at (9,0) {$B$};
  \end{tikzpicture}
  \caption{A graph $G$ where the sink is a cut vertex.\label{fig:s_cut}}
\end{figure}

Now suppose that $s$ is not a cut vertex of $G$. Fix a partition $(A,B)$ of $\V$. Since $s$ is not a cut vertex of $G$, there must exist an edge $(vw)$ of $G$ with $v \in A$ and $w \in B$ (see Figure~\ref{fig:s_not_cut}). In particular, we have $\dgr{A}{w} > 0$, and so $\mathbf{1}^B(w) = 1 - \dgr{1}{w} \leq 0$, so that $\mathbf{1}^B$ cannot be a parking function. This shows that there is no partition $(A,B)$ with respect to which $\mathbf{1}$ is decomposable, and therefore $\mathbf{1} \in \PGPF{G}$, as desired.

\begin{figure}[ht]
 \centering
  \begin{tikzpicture}[scale=0.35]
 \node (0) [draw, rectangle, fill=black] at (6,-3) {};

 \draw [thick] (3,0) circle (1.5cm);
 \draw [thick] (9,0) circle (1.5cm);

 \draw [thick] (4.4,-0.6)--(7.6,-0.6);
 \draw [thick] (4.5,0)--(7.5,0);
 \draw [thick] (4.4,0.6)--(7.6,0.6);
 \draw [thick] (3,-1.5)--(6,-3)--(9,-1.5);

 \node at (3,0) {$A$};
 \node at (9,0) {$B$};
  \end{tikzpicture}
  \caption{A graph $G$ where the sink is not a cut vertex.\label{fig:s_not_cut}}
\end{figure}
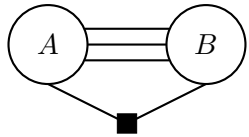
\end{proof}

In the proof above, what stopped $\mathbf{1}^B$ from being a parking function was the fact that it does not take positive values. We now show that this is the only possible issue for general $p^B$.

\begin{proposition}\label{pro:p^B_cond}
Let $p \in \GPF{G}$ be a $G$-parking function, and $(A,B)$ an ordered set partition of $\V$. Then $p$ is decomposable with respect to $(A,B)$ if and only if we have $p^A \in \GPF{G^A}$ and $p^B(v) > 0$ for all $v \in B$.
\end{proposition}

\begin{proof}
That the condition is necessary follows from the fact that parking functions must take positive values. Conversely, suppose that $(A,B)$ is an ordered set partition such that $p^A \in \GPF{G^A}$ and $p^B$ takes positive values. We wish to show that $p^B$ is a $G^B$-parking function. Let $S \subseteq B$. Since $p$ is a $G$-parking function, there exists $v \in B$ such that $p(v) \leq \dgr{V \setminus S}{v}$. By definition, we have:
\begin{align*}
p^B(v) & = p(v) - \dgr{A}{v}\\
 & \leq \dgr{V \setminus S}{v} - \dgr{A}{v}\\
 & = \dgr{V}{v} - \dgr{S}{v} - \dgr{A}{v}\\
 & = \dgr{B \cup \{s\}}{v} + \dgr{A}{v}  - \dgr{S}{v} - \dgr{A}{v}\\
 & = \dgr{B \cup \{s\} \setminus S}{v},
\end{align*}
and therefore $p \in \GPF{G^B}$ as desired.
\end{proof}

For $p \in \GPF{G}$, a \emph{prime decomposition} of $p$ is an ordered set partition $(A_1, \dots, A_k)$ of $\V$ such that for all $i \in [k]$, we have $p^{A_i} \in \PGPF{G^{A_i}}$, where $p^{A_i}$ is defined by $p^{A_i}(v) := p(v) - \dgr{\bigcup\limits_{j=1}^{i-1} A_j}{v}$ for all $v \in A_i$. By successively decomposing a $G$-parking function until no further decompositions are possible, we get the following.

\begin{proposition}\label{pro:prime_decomp}
Every $G$-parking function admits a prime decomposition.
\end{proposition}

\begin{remark}\label{rem:decomp}

Unlike for classical parking functions (or equivalently, Dyck or \luk\ paths), the prime decomposition of a $G$-parking function is not unique. Indeed, consider the $G$-parking function $p = \left(p(v_1), p(v_2), p(v_3), p(v_4)\right) = (1,2,1,1)$ from Figure~\ref{fig:exa_diff_prime_decomp}. Here, $p$ admits two distinct prime decompositions: the partitions $\left( \{v_1\}, \{v_2, v_3, v_4\} \right)$ and $\left( \{v_3, v_4\}, \{v_1,v_2\} \right)$ (in both cases $p^A$ and $p^B$ are the all-1 function $\mathbf{1}$). Notably, for these decompositions, the block sizes are different, which shows that the decomposition is not unique even up to graph symmetries.

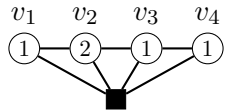
\begin{figure}[ht]
 \centering
  \begin{tikzpicture}[scale=0.271]
 \node (0) [draw, rectangle, fill=black] at (7.5,-2.5) {};
 \node [MyNode, scale=0.8] (1) at (3,0) {1};
 \node [MyNode, scale=0.8] (2) at (6,0) {2};
 \node [MyNode, scale=0.8] (3) at (9,0) {1};
 \node [MyNode, scale=0.8] (4) at (12,0) {1};

 \foreach \xx in {1,2,3,4}
   \foreach \yy in {0}
     \draw [thick] (\xx)--(\yy);
 \draw [thick] (1)--(2)--(3)--(4);

 \node at (3,1.6) {$v_1$};
 \node at (6,1.6) {$v_2$};
 \node at (9,1.6) {$v_3$};
 \node at (12,1.6) {$v_4$};
  \end{tikzpicture}
  \caption{A graphical parking function which admits two different prime decompositions.\label{fig:exa_diff_prime_decomp}}
\end{figure}

\end{remark}

\subsection{Strong recurrence for the ASM}\label{subsec:SR}

It is well known, and straightforward to check, that recurrent configurations of the ASM are closed under grain addition. That is, if $c \in \Rec{G}$, then for any $v \in \V(G)$ such that $c(v) < \dgr{}{v} - 1$, we have $c + \1{v} \in \Rec{G}$, where $\1{v}$ is the indicator function at $v$ (i.e., $\1{v}(w) = 1$ if $w = v$ and $0$ otherwise). In this section, we are interested in the reverse property: which recurrent configurations are closed under grain removal?

Our motivation for this stems from the so-called \emph{minimal recurrent} configurations. A configuration $c \in \Config{G}$ is called minimal recurrent if it is recurrent, and for all $v \in \V(G)$, the configuration $c - \1{v}$ (obtained from the configuration $c$ by removing one grain at the vertex $v$) is \textbf{not} recurrent. Minimal recurrent configurations have been quite widely studied in the literature, and are in bijection with certain acyclic orientations of the underlying graph $G$ (see e.g.~\cite{Schulz}). They have proved to be a key tool in various combinatorial studies of the recurrent configurations on graph families, see for example~\cite{DSSS1, SSS, SelWheel}.

The notion of \emph{strong recurrence} that we introduce in this paper is in some sense orthogonal to that of minimal recurrence. Strongly recurrent configurations will be configurations that remain recurrent after we remove many grains from them. We begin by formally defining the notion.

\begin{definition}\label{def:SR}
Let $G = (\Gamma, s) $ be a graph with vertex set $V = \V  \cup \{s\}$. For a recurrent configuration $c \in \Rec{G}$, we set $V_M(c) := \{v \in N^{V}(s); \, c(v) \geq \dgr{}{v} - \mult{v}{s} \}$. For $v \in V_M(c)$, we define the configuration 
\begin{equation}\label{eq:def_sr} 
c^{v-} :=  c - \sum\limits_{w \in \V \setminus \{v\}} \mult{w}{s} \ccdot \1{w}.
\end{equation}
We say that $c$ is \emph{strongly recurrent} (SR) if for all $v \in V_M(c)$, we have $c^{v-} \in \Rec{G}$. We denote by $\SRec{G}$ the set of SR configurations on $G$.
\end{definition}

In words, the transformation $c \leadsto c^{v-}$ removes grains from all neighbours of the sink except $v$ according to their multiplicity. For many graph families, such  as complete graphs, wheel graphs, fan graphs, complete multi-partite graphs with a dominating sink, and so on, this just means removing exactly one grain from all vertices except $v$. Leaving $c$ unchanged at $v$ is dictated by Dhar's burning criterion (Theorem~\ref{thm:charac_rec}, Conditions~(3) and (4)): $V_M(c)$ is simply the set of vertices which are unstable in the configuration $c + \sum\limits_{w \in \V} \mult{w}{s} \ccdot \1{w}$ (equivalently, the set of possible values for $v_1$ in a burning sequence). In particular, if a configuration $c'$ is obtained from $c$ by removing grains, $c'$ can only be recurrent if at least one of these vertices remains unstable in $c' + \sum\limits_{w \in \V} \mult{w}{s} \ccdot \1{w}$, hence the need to leave $c$ unchanged at some $v \in V_M(c)$. We summarise this discussion through the following useful lemma.

\begin{lemma}\label{lem:V_M(c^v-)}
Let $c \in \SRec{G}$ be a strongly recurrent configuration, and $v \in V_M(c)$. Then we have $V_M\left( c^{v-} \right) = \{ v \}$.
\end{lemma}

\begin{proof}
Since $c^{v-}(v) = c(v)$, and $v \in V_M(c)$, we have $v \in V_M\left( c^{v-} \right)$. Now if $w \neq v$, we have $c^{v-} (w) = c(v) - \mult{v}{s}$, so that $c^{v-}(w) + \mult{v}{s} = c(v) < \dgr{}{v}$ (since $c$ is stable), i.e.\ $w \notin V_M \left( c^{v-} \right)$. This completes the proof.
\end{proof}

\begin{remark}
It is natural to ask what happens if we change the ``for all'' quantifier in Definition~\ref{def:SR} with ``there exists'', i.e.\ if we consider a recurrent configuration $c \in \Rec{G}$ to be strongly recurrent when there \emph{exists} $v \in V_M(c)$ such that $c^{v-} \in \Rec{G}$. Since $c$ is recurrent, the set $V_M(c)$ is non-empty (see discussion above), so the ``exists'' condition is weaker than the ``forall'' one. Moreover, we can see that it is strictly weaker by considering the example $c = \left( c(v_1), c(v_2), c(v_3), c(v_4) \right) = (1, 1, 1, 2)$ from Figure~\ref{fig:counter_exa_SR_simpleG}. In this example, we have $V_M(c) = \{v_1, v_4\}$. We can check that $c^{v_1 -} = (1, 1, 1, 1)$ is not recurrent, since $\{v_2, v_3, v_4\}$ gives a forbidden subconfiguration. On the other hand, $c^{v_4 - } = (0, 1, 1, 2)$ is recurrent by Dhar's burning criterion: starting from $c^{v_4 - }$ and ``toppling the sink'' gives the configuration $(1, 1, 1, 3)$, after which the toppling sequence $v_4, v_3, v_2, v_1$ yields the initial configuration $c^{v_4 - }$. As such, while $c$ is not strongly recurrent in the sense of Definition~\ref{def:SR}, it would be if we replaced the quantifier ``for all'' with ``there exists''.

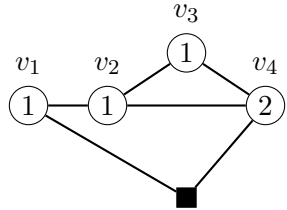
\begin{figure}[ht]
 \centering
  \begin{tikzpicture}[scale=0.35]
   \node (0) [draw, rectangle, fill=black] at (9,-3.5) {};
 \node [MyNode] (1) at (3,0) {1};
 \node [MyNode] (2) at (6,0) {1};
 \node [MyNode] (3) at (9,2) {1};
 \node [MyNode] (4) at (12,0) {2};
 
\draw [thick] (1)--(2);
\draw [thick] (1)--(0);
\draw [thick] (4)--(0);
\draw [thick] (2)--(3)--(4);
\draw [thick] (2)--(4);

 \node at (3,1.5) {$v_1$};
 \node at (6,1.5) {$v_2$};
 \node at (9,3.5) {$v_3$};
 \node at (12,1.5) {$v_4$};
 
  \end{tikzpicture}
  \caption{Illustrating the difference between possible quantifiers in the definition of strong recurrence.
\label{fig:counter_exa_SR_simpleG}}
\end{figure}

\end{remark}

As we saw at the start of this section, the set of recurrent configurations is closed under grain addition. It is therefore also natural to ask whether this also holds for the set of strongly recurrent configurations. That is, given a strongly recurrent configuration $c$ and a vertex $v$ such that $c(v) + 1 < \dgr{}{v}$, is the configuration $c + \1{v}$ also strongly recurrent? 

In fact, for general graphs, this is not the case. Indeed, consider the configuration $c = \left( c(v_1), c(v_2), c(v_3) \right) = (1, 2, 1)$ from Figure~\ref{fig:counter_exa_SR_grain_addition} below. We have $V_M(c) = \{v_3\}$, and $c^{v_3 -} = (0, 2, 1)$ is recurrent by Dhar's burning criterion, so $c$ is strongly recurrent. However, the configuration $c' := c + \1{v_1} = (2, 2, 1)$ is not strongly recurrent. Indeed, $V_M(c') = \{v_1, v_3\}$, but $c'^{v_1 -} = (2, 2, -1)$ takes a negative value at $v_3$, so cannot be recurrent (see Remark~\ref{rem:non-negative}).

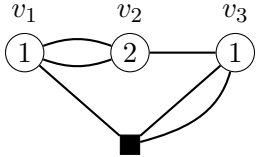
\begin{figure}[ht]
 \centering
  \begin{tikzpicture}[scale=0.35]
   \node (0) [draw, rectangle, fill=black] at (6,-3.5) {};
 \node [MyNode] (1) at (2,0) {1};
 \node [MyNode] (2) at (6,0) {2};
 \node [MyNode] (3) at (10,0) {1};
 
\draw [thick] (2)--(3)--(0)--(1);
\draw [thick, out=20, in=160] (1) to (2);
\draw [thick, out=-20, in=-160] (1) to (2);
\draw [thick, out=15, in=255] (0) to (3);

 \node at (2,1.5) {$v_1$};
 \node at (6,1.5) {$v_2$};
 \node at (10,1.5) {$v_3$};
 
  \end{tikzpicture}
  \caption{A strongly recurrent configuration where strong recurrence is not preserved under grain addition.
\label{fig:counter_exa_SR_grain_addition}}
\end{figure}

We now see that, if we impose a straightforward additional condition on our graph $G$, which is notably satisfied for all \emph{simple} graphs, strong recurrence is in fact preserved under grain addition.

\begin{proposition}\label{pro:SR_grain_addition}
Let $G = (\Gamma, s)$ be a graph such that every edge incident to the sink has multiplicity one, i.e.\ $\mult{w}{s} \leq 1$ for all $w \in \V$. Then the set of strongly recurrent configurations is closed under grain addition. That is, for any $c \in \SRec{G}$ and $v \in \V$ such that $c(v) + 1 < \dgr{}{v}$, we have $c + \1{v} \in \SRec{G}$.
\end{proposition}

\begin{proof}
Suppose that $G$ has no multiple edges incident to the sink, and let $c \in \SRec{G}$ and $v \in \V$ such that $c(v) + 1 < \dgr{}{v}$. 
Define $c_* := c + \1{v}$. Clearly $c_*$ is recurrent, since it results from adding one grain to the recurrent configuration $c$. 
We wish to show that $c_*$ satisfies the strong recurrence property. If $v \notin V_M(c_*)$, then $V_M(c_*) = V_M(c)$, and there is nothing to do: we just use the strong recurrence of $c$ and the fact that recurrence is preserved under grain addition. We may therefore assume that $v \in V_M(c_*)$, i.e.\ that $V_M(c_*) = V_M(c) \cup \{v\}$.

Now, if $w \in V_M(c)$, we have $c_*^{w-} = c^{w-} + \1{v}$, and this is recurrent as above. It therefore remains to show that $c_*^{v-} \in \Rec{G}$. For this, we will use a characterisation of recurrent configurations using \emph{acyclic orientations} of $G$. An \emph{orientation} of $G$ is the assignment of a direction to each edge. An orientation is called \emph{acyclic} if it contains no directed cycles. It is straightforward to check that an acyclic orientation contains at least one \emph{source} vertex, at which every edge is outgoing, and at least one \emph{target} vertex, at which every edge is incoming\footnote{The terminology \emph{souce} is more common than \emph{target} in the literature, but since our graphs already have a designated vertex called the sink, we use the latter here.}. An acyclic orientation $\OO$ of $G$ is called \emph{rooted} if the sink $s$ is the unique target vertex of $\OO$.

Given an orientation $\OO$ of $G$, and a vertex $v \in V$, we denote $\In{v}$, resp.\ $\Out{v}$, the number of incoming, resp.\ outgoing, edges at $v$ in $\OO$ (counted with multiplicity). The result we will use is the following. A configuration $c \in \Config{G}$ is recurrent if and only if there exists a rooted acyclic orientation $\OO$ of $G$ such that:
\begin{equation}\label{eq:orient_comp}
\forall v \in \V, \, c(v) \geq \In{v}.
\end{equation}
We say that such an orientation $\OO$ is \emph{compatible} with $c$. This characterisation was first stated in these terms by Biggs \cite{Biggs}, although the author credits a previous paper \cite{GZ} as having equivalent results.

We use this characterisation to show that $c_*^{v-}$ is recurrent. First, fix some vertex $w \in V_M(c)$. We know that $c_*^{w-}$ is recurrent, so there exists a rooted acyclic orientation $\OO$ compatible with $c_*^{w-}$. Consider the induced orientation $\tilde{\OO}$ on $G[\V]$ that results from deleting $s$ in $\OO$. This orientation remains acyclic, so must have at least one target vertex. By the compatibility condition, any such target vertex $u$ of $\tilde{\OO}$ satisfies:
$$ c_*^{w-}(u) \geq \In{u} = \mathrm{in}^{\tilde{\OO}}(u) = \dgr{\V}{u} = \dgr{V}{u} - \mult{u}{s},$$
i.e.\ $u \in V_M\left( c_*^{w-} \right)$. But by Lemma~\ref{lem:V_M(c^v-)}, this implies $u=w$, i.e.\ $w$ is the only target vertex of the induced orientation $\tilde{\OO}$.

In particular, in the orientation $\OO$, there must be at least one directed path from $v$ to $w$. We consider the orientation $\OO'$ where \emph{all} directed paths from $v$ to $w$ are flipped. It is straightforward to see that $\OO'$ is also a rooted acyclic orientation of $G$, and that $\mathrm{in}^{\OO'}(u) = \In{u}$ for all $u \neq v, w$. We wish to show that $\OO'$ is compatible with $c_*^{v-}$. First, note that since $v, w \in V_M \left( c_* \right)$, the vertices $v$ and $w$ are both neighbours of the sink $s$ in $G$, i.e. $\mult{v}{s} = \mult{w}{s} = 1$. Therefore we have $c_*^{v-} = c_*^{w-} + \1{v} - \1{w}$. 

By preceding remarks, it is sufficient to check the compatibility condition~\eqref{eq:orient_comp}  for $\OO'$ at $v$ and $w$. For $v$, we note that since $v \in V_M \left( c_*^{v-} \right)$, we have $c_*^{v-} \geq \dgr{}{v} - \mult{v}{s} = \dgr{}{v} - 1$. 
Moreover, since $s$ is the only target of $\OO'$, we have $\mathrm{in}^{\OO'}(v) \leq \dgr{}{v} - 1$, and the compatibility condition is ensured at $v$. Now note that in the transformation from $\OO$ to $\OO'$, at least one incoming edge at $w$ has had its direction flipped, i.e. $\mathrm{in}^{\OO'}(w) \leq \In{w} - 1$. By combining this with the compatibility condition for $\OO$, and the equality $c_*^{v-} = c_*^{w-} + \1{v} - \1{w}$, we get:
$$ c_*^{v-}(w) = c_*^{w-}(w) - 1 \geq \In{w} - 1 \geq \mathrm{in}^{\OO'}(w), $$
i.e.\ the compatibility condition holds at $w$. Therefore $\OO'$ is compatible with $c_*^{v-}$, so that $c_*^{v-}$ is indeed recurrent. This concludes the proof.
\end{proof}

\subsection{The main result}\label{subsec:main_res}

We now state and prove the main result of this paper.

\begin{theorem}\label{thm:bij_sr_GPPF}
For any graph $G$, the map $ c \mapsto \dfunc - c $ is a bijection from $\SRec{G}$ to $\PGPF{G}$.
\end{theorem}

\begin{proof}
The map $c \mapsto \dfunc - c$ is clearly bijective (it is an involution). It therefore suffices to show that $c$ is strongly recurrent if and only if $p(c) := \dfunc - c$ is a prime $G$-parking function.

First, suppose that $c \in \Rec{G}$ is \emph{not} strongly recurrent. Then by definition there exists $v \in \V$ such that $c(v) + \mult{v}{s} \geq \dgr{V}{v}$ (i.e., $v \in V_M(c)$), but $c^{v-} \notin \Rec{G}$. Let $F \subseteq \V$ be a maximal forbidden sub-configuration for $c^{v-}$, i.e., $c^{v-}(w) < \dgr{F}{w}$ for all $w \in F$. We claim that we have $v \notin F$, so $A := \V \setminus F \neq \emptyset$ and is therefore a proper subset of $\V$. Indeed, we have:
\begin{align*}
c^{v-}(v) = c(v) &\geq \dgr{V}{v} - \mult{v}{s} \\ 
&= \dgr{\V}{v} \geq \dgr{F}{v},
\end{align*} 
so $v \notin F$, as desired. We now show that $p := \dfunc - c$ is decomposable with respect to the partition $(A,F)$.

By maximality of $F$, for $c^{v-}$ we can ``burn down to $F$'' in Dhar's burning algorithm. That is, starting from $c^{v-} + \sum\limits_{w \in \V} \mult{w}{s} \ccdot \1{w}$, we can topple all vertices in $A$ in some order. Since $c$ is obtained from $c^{v-}$ through adding grains, this remains a legal toppling sequence for $c + \sum\limits_{w \in \V} \mult{w}{s} \ccdot \1{w}$ (albeit one that does not result in stabilisation yet). In other words, there exists a sequence $v_0 = s, v_1 = v, \ldots, v_{\vert A \vert}$ such that $A = \{v_1, \dots, v_{\vert A \vert} \}$, and $c(v_i) \geq \dgr{V \setminus \{v_0, \ldots,v_{i-1}\}}{v_i}$ for all $1 \leq i \leq \vert A \vert$. 

Now consider $c^A := \dfunc^{A \cup \{s\}} - p^A \in \Config{G^A}$, and fix $w \in A$. Since $c$ is stable, we have $c(w) < \dgr{V}{w}$. Recall that $c(w) = \dgr{V}{w} - p(w)$, so that:
\begin{align*}
c^A(w) &= \dgr{A \cup \{s\}}{w} - p^A(w) \\
 & = \dgr{A \cup \{s\}}{w} - p(w) \\
 & = \dgr{V}{w} - p(w) - \left(\dgr{V}{w} - \dgr{A \cup \{s\}}{w} \right) \\
 &= c(w) - \dgr{F}{w} < \dgr{V}{w} - \dgr{F}{w} = \dgr{A \cup \{s\}}{w}, 
\end{align*}
which shows that $c^A$ is stable on $G^A$. Now by construction, for $i = 1, \dots, \vert A \vert$, we have: 
$$
c^A(v_i) = c(v_i) - \dgr{F}{v_i}
\geq \dgr{V \setminus \{v_0, \ldots, v_{i-1}\}}{v_i} - \dgr{F}{v_i}
= \dgr{A \cup \{s\} \backslash \{v_0, \ldots, v_{i-1}\}}{v_i},$$
so that $c^A \in \Rec{G^A}$ by Dhar's burning criterion (Theorem~\ref{thm:charac_rec}, Condition~(4)). Theorem~\ref{thm:bij_rec_GPF} then implies that $p^A \in \GPF{G^A}$.

We now consider $p^F := p - \dfunc^{A}$. Fix $w \in F$. Recall that $F$ is a forbidden subconfiguration for $c^{v-}$, and that $v \notin F$, so that $c^{v-}(w) = c(w) - \mult{w}{s} < \dgr{F}{w}$, i.e.\ $c(w) < \dgr{F \cup \{s\}}{w}$. Since $p(w) = \dgr{V}{w} - c(w)$, this yields:
\begin{align*}
p^F(w) &= p(w) - \dgr{A}{w} \\
&= \dgr{V}{w} - c(w) - \dgr{A}{w} \\
&= \dgr{F \cup \{s\}}{w} - c(w) >0.
\end{align*}
Finally, for the parking function $p \in \GPF{G}$, we have constructed an ordered set partition $(A,F)$ of $\V$ such that $p^A$ is a $G^A$-parking function and $p^F(w) > 0$ for all $w \in F$. From Proposition~\ref{pro:p^B_cond}, we deduce that $p$ is decomposable with respect to $(A,F)$, i.e.\ that $p$ is not prime, as desired.

We now show the converse. Suppose $p \in \GPF{G}$ is \emph{not} prime, let $(A,B)$ be a partition of $\V$ with respect to which $p$ is decomposable, i.e.\ $p^A \in \GPF{G^A}$ and $p^B \in \GPF{G^B}$. Let $c := \dfunc - p \in \Rec{G}$ be the corresponding sandpile configuration. We wish to show that $c$ is not strongly recurrent, i.e.\ that there exists $v \in V_M(c)$ such that $c^{v-}$ is not recurrent. 

We first claim that $V_M(c) \cap A \neq \emptyset$. Define $c^A := \dfunc^{A \cup \{s\}} - p^A \in \Config{G^A}$. Since $p^A$ is a $G^A$-parking function, $c^A$ is recurrent. In particular, there exists $v \in A$ such that $c^A(v) + \mult{v}{s} \geq \dgr{A \cup \{s\}}{v}$ (for this, choose $v$ to be any vertex $v_1$ in a burning sequence for $c^A$, see Theorem~\ref{thm:charac_rec}, Condition~(4)), so that:
\begin{align*}
c(v) &= \dgr{V}{v} - p(v) = \dgr{V}{v} - p^A(v)\\
 & = \dgr{A \cup \{s\}}{v} - p^A(v) + \dgr{B}{v} \\ 
 & = c^A(v) + \dgr{B}{v}\\ 
 & \geq \dgr{A \cup \{s\}}{v} - \mult{v}{s} + \dgr{B}{v} = \dgr{V}{v} - \mult{v}{s},
\end{align*} 
which implies $v \in V_M(c)$, as desired.

Finally, we want to show that $B$ is a forbidden subconfiguration for $c^{v-}$. For all $w \in B$, we have $c^{v-}(w) = c(w) - \mult{w}{s} = \dgr{V}{w} - \mult{w}{s} - p(w)$, and by using $p(w) = p^B(w) + \dgr{A}{w}$ and $p^B(w) > 0$, we obtain:
\begin{align*}
c^{v-}(w) & = \dgr{V}{w} - \mult{w}{s} - p(w) \\
 & = \dgr{V}{w} - \mult{w}{s} - \left( p^B(w) + \dgr{A}{w} \right) \\
 & = \dgr{B \cup \{s\}}{w} - \mult{w}{s} - p^B(w) \\
 & < \dgr{B \cup \{s\}}{w} - \mult{w}{s} = \dgr{B}{w},
\end{align*}
Since this holds for all $w \in B$, this means exactly that $B$ is a forbidden subconfiguration for $c^{v-}$ as desired, which implies that $c^{v-}$ is not recurrent by Theorem~\ref{thm:charac_rec}, i.e.\ that $c$ is not strongly recurrent. This completes the proof.
\end{proof}

By transferring the definition of strong recurrence to parking functions via this bijection, we get the following straightforward consequence, which gives an alternate characterisation of prime $G$-parking functions.

\begin{corollary}\label{cor:GPPF_charac}
Let $p \in \GPF{G}$, and define $V_M(p) := \{v \in \V; \, p(v) \leq \mult{v}{s} \}$.  Then $p$ is prime if and only if, for all $v \in V_M(p)$, we have $p^{v+} := p + \sum\limits_{w \in \V \setminus \{v\} } \mult{w}{s} \ccdot \1{w} \in \GPF{G}$.
\end{corollary}


\section{Applications to graph families}\label{sec:SR_graphs}

In this section, we study parking functions and strongly recurrent configurations on graph families with high degrees of symmetry. We begin with a useful technical lemma. For a graph $G$ and vertex $v \in \V(G)$, the map $\Del{v}$ maps a function $p : \V(G) \rightarrow \N$ to the function $p' : \V(G \setminus \{v\}) \rightarrow \N$, where $p'$ is simply the restriction of $p$ to vertices of $G \setminus \{v\}$.

\begin{lemma}\label{lem:sym_PPFtoPF}
Let $G = (\Gamma, s)$ be a graph, and $v \in \V(G)$ be a non-sink vertex such that for all $w \neq s, v$, we have $\mult{w}{s} = \mult{w}{v}$. Let $p \in \PGPF{G}$ be a prime $G$-parking function such that $v \in V_M(p)$. Then we have $\Del{v}(p) \in \GPF{G \setminus \{v\}}$.
\end{lemma}

\begin{proof}
To simplify notation, we write $V = V(G)$ and $\V = \V(G)$. We then use $'$ to indicate removal of the vertex $v$, i.e.\ $G' = G \setminus \{v\}$, $V' = V(G') = V \setminus \{v\}$, $\V' = V' \setminus \{s\}$, and $p' = \Del{v}(p)$. We wish to show that $p'$ is a $G'$-parking function. Let $S \subseteq \V'$. We want to find $w \in S$ such that $p'(w) = p(w) \leq \dgr{V' \setminus S}{w}$. But $p$ is a prime $G$-parking function, so by Corollary~\ref{cor:GPPF_charac} we have $p^{v+} \in \GPF{G}$. Since $S \subseteq \V$, there exists $w \in S$ such that $p^{v+}(w) \leq \dgr{V \setminus S}{w}$. We re-write this as $p(w) + \mult{w}{s} \leq \dgr{V' \cup \{v\} \setminus S}{w} = \dgr{V' \setminus S}{w} + \mult{w}{v}$, and the result immediately follows from the assumption of the lemma.
\end{proof}

\begin{remark}\label{rem:del_cond}
The condition on $v$ in Lemma~\ref{lem:sym_PPFtoPF} may seem a little non-intuitive, but it holds for example when the transposition $(vs)$ is an isomorphism of $\Gamma$, which is the case for a number of highly symmetrical graph families.
\end{remark}

\subsection{Warm-up: wheel and complete graph cases}\label{subsec:complete}

As a warm-up, we consider the cases where $G$ is a wheel or complete graph. The wheel graph $W_n^0$ is the graph consisting of a cycle on vertices $[n]$ and an additional central vertex $0$ connected to each vertex of the cycle through a single edge (see Figure~\ref{fig:wheel}). The central vertex $0$ is the sink of the graph. For convenience, we write $c = (c_1, \dots, c_n) = (c(1), \dots, c(n))$ for a configuration $c \in \Config{W_n^0}$.

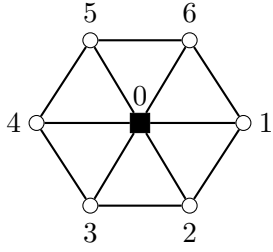
\begin{figure}[ht]
  \centering
  \begin{tikzpicture}[scale=0.55]
  \sink
  \node [MyNode] (1) at (2.5,0) {};
  \node [MyNode] (2) at (1.2,-2) {};
  \node [MyNode] (3) at (-1.2,-2) {};
  \node [MyNode] (4) at (-2.5,0) {};
  \node [MyNode] (5) at (-1.2,2) {};
  \node [MyNode] (6) at (1.2,2) {};
  \foreach \xx in {1,...,6}
    \draw [thick] (0)--(\xx);
  \draw [thick] (1)--(2)--(3)--(4)--(5)--(6)--(1);
  \node [above, yshift=3pt] at (0) {$0$};
  \node [right, xshift=2pt] at (1) {$1$};
  \node [below, yshift=-3pt] at (2) {$2$};
  \node [below, yshift=-3pt] at (3) {$3$};
  \node [left, xshift=-2pt] at (4) {$4$};
  \node [above, yshift=3pt] at (5) {$5$};
  \node [above, yshift=3pt] at (6) {$6$};
  \end{tikzpicture}
  \caption{The wheel graph $W_6^0$.\label{fig:wheel}}
\end{figure}

In our setting, a stable non-negative configuration is a word $c \in \{0, 1, 2\}^n$. Recurrent configurations have been well-studied on wheel graphs, see~\cite{SelWheel} and references therein: a configuration $c$ is recurrent if and only if $c$ contains at least one occurrence of $2$ and, for any indices $i, j \in [n]$ such that $c_i = c_j = 0$, there exists $k$ in the cyclic interval $(i,j)^{\mathrm{cyc}}$ such that $c_k = 2$. Here, the cyclic interval $(i,j)^{\mathrm{cyc}}$ is simply the standard interval $(i,j)$ if $i < j$, and the union $(i, n] \cup [1, j)$ if $i \geq j$.

Now let us examine the strong recurrence condition for such $c$. First, note that $V_M(c) = \{i \in [n]; \, c_i = 2\}$, and that if $i \in V_M(c)$, we have $c^{i-} = c - \sum\limits_{\begin{footnotesize}\begin{array}{c} j \in [n]\\ j \neq i \end{array}\end{footnotesize}} \1{j}$. In particular, if $c$ is strongly recurrent it must not have any vertex with $0$ grains, i.e.\ we must have $c \in \{1, 2\}^n$. Moreover, we claim that there must exist at most one index $j \in [n]$ such that $c_j = 1$. Indeed, suppose there exist $j \neq j'$ with $c_j = c_{j'} = 1$, and fix $i \in V_M(c)$, i.e.\ $c_i = 2$. Then one of the cyclic intervals $(j,j')^{\mathrm{cyc}}$ and $(j',j)^{\mathrm{cyc}}$ does not contain $i$, and since $i$ is the only vertex with $2$ grains in $c^{i-}$, the configuration $c^{i-}$ cannot be recurrent by the above characterisation. We have therefore shown that if $c \in \SRec{W_n^0}$, then we have $c \in \{1,2\}^n$ and $\left\vert \{ i \in [n]; \, c_i = 1\} \right\vert \leq 1$. It is straightforward to show the converse, namely that if a configuration has two grains at each vertex, except perhaps one vertex where it has one grain, then it is strongly recurrent. We summarise these findings in the following proposition.

\begin{proposition}\label{pro:wheel_sr}
Let $c = (c_1, \dots, c_n) \in \Config{W_n^0}$. Then $c$ is strongly recurrent if and only if we have $c \in \{1,2\}^n$ and $\left\vert \{ i \in [n]; \, c_i = 1\} \right\vert \leq 1$. In particular, we have $\left\vert \SRec{W_n^0} \right\vert = \left\vert \PGPF{W_n^0} \right\vert = n+1$.
\end{proposition}

We now consider the case where $G = K_n^0$ is the complete graph with vertex set $[n]_0 := [n] \cup \{0\}$ and sink $0$. Recall from Proposition~\ref{pro:GPF_PF_comp} that the set of $K_n^0$-parking functions is exactly the set of classical parking functions in this case. 

\begin{proposition}\label{pro:classical_PPF_complete_graph}
We have $\PPF{n} = \PGPF{K_n^0}$. In particular, we have $\left\vert \PGPF{K_n^0} \right\vert = \left\vert \SRec{K_n^0} \right\vert = (n-1)^{n-1}$.
\end{proposition}

\begin{proof}
If $p = (p_1, \dots, p_n) \in \PF{n}$ is \emph{not} prime, then it has some breakpoint $1 \leq j < n$ such that $\left\vert \{i \in [n]; \, p_i \leq j \} \right\vert = j$. It is straightforward to check that if $A := \{i \in [n]; \, p_i \leq j \}$, then $p$ is decomposable with respect to the partition $(A, [n] \setminus A)$ in the sense of Definition~\ref{def:prime_GPF}. Conversely, if $p \in \GPF{K_n^0}$ is decomposable with respect to a partition $(A,B)$, then $\vert A \vert$ is a breakpoint for $p$. The enumeration was discussed previously in Section~\ref{subsec:PPF}.
\end{proof}

Throughout the remainder of this section, it will be natural to consider parking functions up to the symmetries of the graph. We say that a parking function $p = (p_1, \dots, p_n)$ is \emph{non-decreasing} if $p_1 \leq \dots \leq p_n$, and denote $\incPF{n}$, resp.\ $\incPPF{n}$ the set of non-decreasing parking functions, resp.\ non-decreasing prime parking functions of size $n$. Similarly, we denote $\decRec{K_n^0}$, resp.\ $\decSR{K_n^0}$, the set of non-increasing recurrent, resp.\ non-increasing strongly recurrent, configurations on $K_n^0$, with which non-decreasing (prime) parking functions are in bijection.

\begin{proposition}\label{pro:incPPF_complete_graph}
There is a bijection from $\incPPF{n}$ to $\incPF{n-1}$. In particular, we have $\left\vert \incPPF{n} \right\vert = \left\vert \decSR{K_n^0} \right\vert = \mathrm{Cat}_{n-1}$, where $\mathrm{Cat}_k := \frac{1}{k+1} \binom{2k}{k}$ is the $k$-th Catalan number.
\end{proposition}

\begin{proof}
This can be easily seen via the bijection between non-decreasing parking functions to Dyck paths from Remark~\ref{rem:lattice_paths}. A prime parking function corresponds to a Dyck path which only intersects the $x$-axis at its beginning and end points. Removing the first and last steps of such a path yields a Dyck path of size one less, and this operation is clearly bijective.
\end{proof}

\begin{remark}\label{rem:PPF_remove_1}
The operation of removing the first and last steps of the Dyck path corresponds to removing $p_1 = 1$ from the associated non-decreasing parking function, yielding a parking function $p' = (p_2, \dots, p_n) \in \PF{n-1}$. In fact, the property ``remove a $1$ entry to get a parking function of size one less'' is often taken as a definition of prime parking functions (see~\cite[Exercise~5.49]{StanEC}). Lemma~\ref{lem:sym_PPFtoPF} shows that this property holds for general prime parking functions on ``symmetrical'' graphs. We will see in later subsections that the property in fact fully characterises prime parking functions (i.e., can be used as an alternate definition as in the classical case) on several other graph families.
\end{remark}

\subsection{The complete tripartite case}\label{subsec:tripartite}

In this section, we consider the complete tripartite graph $K_{p,q}^0$. This graph consists of three independent sets: $P = \{v_1^p, v_2^p, \ldots, v_p^p\}$, $Q = \{v_1^q, v_2^q, \ldots, v_q^q\}$, and $S = \{v_0\}$ (the sink vertex), with a single edge connecting any pair of vertices in different sets (see Figure~\ref{fig:tripartite}). We will see that our notion of prime parking functions on such graphs extends the notion of prime $(p,q)$-parking functions defined by Armon \emph{et al.}~\cite{ABHKMMY}. To unify and simplify notation, we write $p = (p^p; p^q) = (p^p_1,\dots, p^p_p; p^q_1,\dots p^q_q) = \left( p(v_1^p), \dots, p(v^p_p); p(v_1^q), \dots, p(v^q_q) \right) $ for a $K_{p,q}^0$-parking function. We use analogous notation for the corresponding recurrent configuration $c := \dfunc - p \in \Rec{K_{p,q}^0}$.

\begin{figure}[ht]
 \centering
  \begin{tikzpicture}[scale=0.35]
 \begin{scope}[shift={(0,-3)}] \sink \end{scope}
 \node [MyNode] (1) at (3,0) {};
 \node [MyNode] (2) at (6,0) {};
 \node [MyNode] (3) at (9,0) {};
 \node [MyNode] (4) at (12,0) {};
 \node [MyNode] (5) at (15,0) {};
 
 \node [MyNode] (6) at (4.5,-6) {};
 \node [MyNode] (7) at (7.5,-6) {};
 \node [MyNode] (8) at (10.5,-6) {};
 \node [MyNode] (9) at (13.5,-6) {};

 \foreach \xx in {1,...,9}
   \foreach \yy in {0}
     \draw [thick] (\xx)--(\yy);
     
 \foreach \xx in {1,...,5}
   \foreach \yy in {6,...,9}
     \draw [thick] (\xx)--(\yy);

 \foreach \xx in {1,...,5}
   \node [above, yshift=2pt] at (\xx) {$v_{\xx}^p$};
 \foreach \yy [evaluate=\yy as \zz using {int(\yy-5)}] in {6,...,9} {
   \node [below, yshift=-2pt] at (\yy) {$v_{\zz}^q$};
 }
  \end{tikzpicture}
  \caption{The complete tripartite graph $K_{5,4}^0$.\label{fig:tripartite}}
\end{figure}
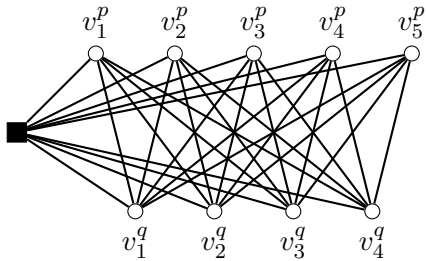

Cori and Poulalhon~\cite{CorPou} introduced a notion of $(p,q)$-parking functions to characterise recurrent configurations of the ASM on $K_{p,q}^0$. We give here an equivalent definition due to Snider and Yan~\cite{YanUpq}. Given two \emph{non-decreasing} vectors $\mathbf{a} = (a_1, \dots, a_p) \in \left([q] \cup \{0\}\right)^p, \mathbf{b} = (b_1, \dots, b_q) \in \left( [p] \cup \{0\} \right)^q$, we define two lattice paths $\LL^{\perp}_{\mathbf{a}}$ and $\LL_{\mathbf{b}}$, from $(0,0)$ to $(p,q)$ with steps $E := (1,0)$ and $N := (0,1)$, as follows:
\begin{itemize}
\item The path $\LL^{\perp}_{\mathbf{a}}$ has its $i$-th $E$ step at $y$-coordinate $a_i$ for each $i \in [p]$, i.e.~ $\LL^{\perp}_{\mathbf{a}} =  N^{a_1} E N^{a_2 - a_1} E \dots $;
\item The path $\LL_{\mathbf{b}}$ has its $j$-th $N$ step at $x$-coordinate $b_j$ for each $j \in [q]$, i.e.~ $\LL_{\mathbf{b}} = E^{b_1} N E^{b_2 - b_1} N \dots $;
\end{itemize}
Figure~\ref{fig:ab_path} shows an example of this construction. Given two lattice paths $\LL, \LL'$, we say that $\LL$ lies \emph{weakly above} $\LL'$ if for any $i \in [p]$, the $y$-coordinate of the $i$-th $E$ step in $\LL$ is greater than or equal to the $y$-coordinate of the $i$-th $E$ step in $\LL'$.

\begin{figure}[ht]
 \centering
 \begin{tikzpicture}[scale=0.6]
 \draw (0,0) grid (5,4);
 \draw [blue, very thick] (0,0)--(0,2)--(1,2)--(1,3)--(2,3)--(2,4)--(5,4);
 \node [blue] at (0.5,2.5) {$\mathcal{L}_{\mathbf{b}}$};
 \node [blue] at (-0.5,0.5){0};
 \node [blue] at (-0.5,1.5){0};
 \node [blue] at (-0.5,2.5){1};
 \node [blue] at (-0.5,3.5){2};
  
 \draw [purple, very thick] (0,0)--(2,0)--(2,2)--(4,2)--(4,3)--(5,3)--(5,4);
 \node [purple] at (2.6,0.6) {$\mathcal{L}^{\perp}_{\mathbf{a}}$};
 \node [purple] at (0.5,-0.5){0};
 \node [purple] at (1.5,-0.5){0};
 \node [purple] at (2.5,-0.5){2};
 \node [purple] at (3.5,-0.5){2}; 
 \node [purple] at (4.5,-0.5){3}; 
 \end{tikzpicture}
  \caption{The lattice paths corresponding to the vectors $\mathbf{a} = (0,0,2,2,3)$, $\mathbf{b} = (0,0,1,2)$; here we have $p = 5$ and  $q = 4$.\label{fig:ab_path}}
\end{figure}
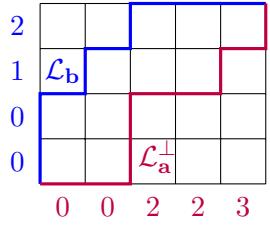

\begin{definition}[{\cite[Theorem~3.4]{YanUpq}, see also~\cite[Definition~2.4]{ABHKMMY}}]\label{def:pq_ab}
Let $p = (p^p; p^q) = (p^p_1,\dots, p^p_p; p^q_1,\dots p^q_q) = \left( p(v_1^p), \dots, p(v^p_p); p(v_1^q), \dots, p(v^q_q) \right) $ with $p^p \in [q+1]^p$ and $p^q \in [p+1]^q$. We say that $p$ is a \emph{$(p,q)$-parking function} if the lattice path $\LL_{\inc{p^q} - \func{1}}$ lies weakly above the lattice path $\LL^{\perp}_{\inc{p^p} - \func{1}}$. We denote by $\PF{p,q}$ the set of $(p,q)$-parking functions.
\end{definition}

Note that unlike in~\cite{CorPou, YanUpq, ABHKMMY}, here we need to subtract $1$ from the (re-arranged) entries of the vectors $p^p$ and $p^q$ to construct the lattice paths (see also Remark~\ref{rem:GPF_0_1}).

\begin{example}\label{ex:pq_PF}
Consider the vectors $\mathbf{a} = (0,0,2,2,3)$ and $\mathbf{b} = (0,0,1,2)$ from Figure~\ref{fig:ab_path}, with $p=5$ and $q=4$. Add $1$ to each entry to get $\mathbf{a}' = (1,1,3,3,4)$ and $\mathbf{b}' = (1,1,2,3)$. Since $\LL_{\mathbf{b}}$ lies weakly above $\LL^{\perp}_{\mathbf{a}}$, any re-arrangement of $\mathbf{a}'$ and $\mathbf{b}'$ will yield a parking function. For example, we have $(3,4,1,1,3;1,3,2,1) \in \PF{5,4}$.
\end{example}

Cori and Poulalhon~\cite[Proposition~7]{CorPou} showed that the map $ \PF{p,q} \rightarrow \Rec{K_{p,q}^0}, p \mapsto c(p) := \dfunc - p$ is a bijection. This immediately yields the following.

\begin{proposition}\label{pro:GPF_PF_Tri}
We have $\PF{p,q} = \GPF{K_{p,q}^0}$.
\end{proposition}

Armon \emph{et al.}~\cite{ABHKMMY} defined a notion of primeness for $(p,q)$-parking functions. Similarly to Definition~\ref{def:pq_ab}, we choose to define this notion using an equivalent characterisation in terms of lattice paths (see~\cite[Proposition~4.2]{ABHKMMY}).

\begin{definition}\label{def:prime_pq_ab} 
Let $p = (p^p, p^q) \in \PF{p,q}$ be a $(p,q)$-parking function, and define $\mathbf{a} := \inc{p^p} - \func{1}$ and $\mathbf{b} := \inc{p^q} - \func{1}$ be the corresponding sequences defining the lattice paths $\LL^{\perp}_{\mathbf{a}}$ and $\LL_{\mathbf{b}}$ respectively. We say that $p$ is \emph{prime} if $\LL^{\perp}_{\mathbf{a}}$ and $\LL_{\mathbf{b}}$ only intersect each other at their end points $(0,0)$ and $(p,q)$. We let $\PPF{p,q}$ denote the set of prime $(p,q)$-parking functions.
\end{definition}

We are now equipped to state the main result of this section, which shows that the notion of prime $G$-parking functions introduced in this paper is a generalisation of that introduced in~\cite{ABHKMMY}.

\begin{theorem}\label{thm:tripartite_pqPF}
We have $\PPF{p,q} = \PGPF{K_{p,q}^0}$. In particular, we have $\left\vert \PGPF{K_{p,q}^0} \right\vert = \left\vert \PPF{p,q} \right\vert = p^q (q-1)^{p-1} + q^p (p-1)^{q-1} - (p+q-1)(p-1)^{q-1}(q-1)^{p-1}$.
\end{theorem}

\begin{proof}
Let $p = (p^p; p^q) \in \PPF{p,q}$ be a prime $(p,q)$-parking function, with associated lattice paths $\LL^{\perp}_{\mathbf{a}}$ and $\LL_{\mathbf{b}}$. By Proposition~\ref{pro:GPF_PF_Tri}, we have $p \in \GPF{K_{p,q}^0}$, since $p$ is also a $(p,q)$-parking function. Moreover, by construction, we have $V_M(p) = \{i \in [p]; \, p_i^p = 1\} \cup \{j \in [q]; \, p_j^q = 1\}$. We wish to apply Corollary~\ref{cor:GPPF_charac}. For this, fix $i \in [p]$ such that $p_i^p = 1$ (the case $j \in [q]$ is analogous). By symmetry, and to simplify notation, we assume that $i=1$. We wish to show that $p^{1+} \in \GPF{K_{p,q}^0}$. 

By construction, we have $p^{1+} = p + \func{1} - \1{1}$. We now consider the effects of the transformation $p \leadsto p^{1+}$ on the lattice paths. First, note that since $p$ is prime, the lattice paths $\LL^{\perp}_{\mathbf{a}}$ and $\LL_{\mathbf{b}}$ only intersect at their start and end points. In particular, $\LL^{\perp}_{\mathbf{a}}$ starts with an $E$ step and ends with an $N$ step, while $\LL_{\mathbf{b}}$ starts with an $N$ step and ends with an $E$ step. We write $\LL^{\perp}_{\mathbf{a}} = E L_a N$ and $\LL_{\mathbf{b}} = N L_b E$. Now, in the transformation $p \leadsto p^{1+}$, other than the initial $1$ entry in the $\mathbf{a}$ vector, all other entries in $\mathbf{a}$ and $\mathbf{b}$ increase by one. This transforms $\LL^{\perp}_{\mathbf{a}}$ into $\LL^{\perp}_{\mathbf{a'}} := ENL_a$ and $\LL_{\mathbf{b}}$ into $\LL_{\mathbf{b'}} := ENL_b$, as shown on Figure~\ref{fig:path_transform}. We wish to show that $\LL_{\mathbf{b'}}$ lies weakly above $\LL^{\perp}_{\mathbf{a'}}$.

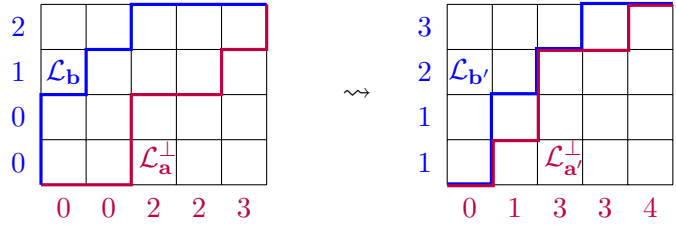
\begin{figure}[ht]
 \centering
 \begin{tikzpicture}[scale=0.6]
 \draw (0,0) grid (5,4);
 \draw [blue, very thick] (0,0)--(0,2)--(1,2)--(1,3)--(2,3)--(2,4)--(5,4);
 \node [blue] at (0.5,2.5) {$\mathcal{L}_{\mathbf{b}}$};
 \node [blue] at (-0.5,0.5){0};
 \node [blue] at (-0.5,1.5){0};
 \node [blue] at (-0.5,2.5){1};
 \node [blue] at (-0.5,3.5){2};
 \draw [purple, very thick] (0,0)--(2,0)--(2,2)--(4,2)--(4,3)--(5,3)--(5,4);
 \node [purple] at (2.6,0.6) {$\mathcal{L}^{\perp}_{\mathbf{a}}$};
 \node [purple] at (0.5,-0.5){0};
 \node [purple] at (1.5,-0.5){0};
 \node [purple] at (2.5,-0.5){2};
 \node [purple] at (3.5,-0.5){2}; 
 \node [purple] at (4.5,-0.5){3}; 
 
 \node at (7,2) {$\leadsto$};
 
 \begin{scope}[shift={(9,0)}]
 \draw (0,0) grid (5,4);
 \begin{scope}[shift={(-0.02,0.02)}]
 \draw [blue, very thick] (0.02,0)--(1,0)--(1,2)--(2,2)--(2,3)--(3,3)--(3,4)--(5.02,4);
 \end{scope}
 \node [blue] at (0.5,2.5) {$\mathcal{L}_{\mathbf{b'}}$};
 \node [blue] at (-0.5,0.5){1};
 \node [blue] at (-0.5,1.5){1};
 \node [blue] at (-0.5,2.5){2};
 \node [blue] at (-0.5,3.5){3};
 \begin{scope}[shift={(0.02,-0.02)}]
 \draw [purple, very thick] (-0.02,0)--(1,0)--(1,1)--(2,1)--(2,3)--(4,3)--(4,4)--(4.98,4);
 \end{scope}
 \node [purple] at (2.6,0.6) {$\mathcal{L}^{\perp}_{\mathbf{a'}}$};
 \node [purple] at (0.5,-0.5){0};
 \node [purple] at (1.5,-0.5){1};
 \node [purple] at (2.5,-0.5){3};
 \node [purple] at (3.5,-0.5){3}; 
 \node [purple] at (4.5,-0.5){4}; 
 \end{scope}
 \end{tikzpicture}
  \caption{Illustrating the effects of the transformation $p \leadsto p^{v+}$ on the associated lattice paths, where $v \in V_M(p) \cap P$.\label{fig:path_transform}}
\end{figure}

Seeking contradiction, suppose there exists some $i$ such that if $y_a$ and $y_b$ denote the $y$-coordinates of the $i$-th $E$ steps in $\LL^{\perp}_{\mathbf{a'}}$ and $\LL_{\mathbf{b'}}$ respectively, we have $y_a > y_b$. Note that necessarily we must have $i > 1$ since the first $E$-step in $\LL^{\perp}_{\mathbf{a'}}$ has $y$-coordinate equal to zero. Now by construction, the $i$-th $E$ step in $\LL^{\perp}_{\mathbf{a}}$ has $y$-coordinate $y_a - 1$, while the $(i-1)$-th $E$ step in $\LL_{\mathbf{b}}$ has $y$-coordinate $y_b \leq y_a - 1$. Since $\LL_{\mathbf{b}}$ lies weakly above $\LL^{\perp}_{\mathbf{a}}$ (because $p$ is a $(p,q)$-parking function), this implies that at the $x$-coordinate $i-1$, the paths $\LL^{\perp}_{\mathbf{a}}$ and $\LL_{\mathbf{b}}$ must intersect, which contradicts the primeness of $p$.

Conversely, let $p = (p^p; p^q) \in \PGPF{K_{p,q}^0}$ be a prime $K_{p,q}^0$-parking function. As previously, in particular $p$ is a $K_{p,q}^0$-parking function, and therefore a $(p,q)$-parking function by Proposition~\ref{pro:GPF_PF_Tri}. Let $\LL^{\perp}_{\mathbf{a}}$ and $\LL_{\mathbf{b}}$ be the lattice paths associated with $p$, with $\LL_{\mathbf{b}}$ therefore lying weakly above $\LL^{\perp}_{\mathbf{a}}$. We wish to show that these two paths do not intersect other than at their end points. To simplify notation, we assume (by symmetry) that $p^p$ and $p^q$ are non-decreasing.

First, note that we must have $p^p_1 = p^q_1 = 1$. Indeed, suppose that $p^p_1 > 1$. Since $V_M(p)$ is non-empty, we must then have $p^q_1 = 1$. But since $p \in \PGPF{K_{p,q}^0}$ is prime, $p' := p^{v_1^q + }$ is a $K_{p,q}^0$-parking function. However, we have $p'^p_i = p^p_i + 1 > 2$ for all $i \in [p]$ and $p'^q_j = p^q_j + 1 > 1$ for all $j \in [q] \setminus \{1\}$. Taking $S = V \setminus \{0, v^q_1\}$ then yields a contradiction with Definition~\ref{def:GPF}. The proof that $p^q_1 = 1$ is analogous. Moreover, by Remark~\ref{rem:GPF_upper_bound} and Corollary~\ref{cor:GPPF_charac}, this further implies that $p^p_p \leq q$ and $p^q_q \leq p$ (we must be able to add $1$ to these values and still be less than or equal to the degree of the corresponding vertex). As such, the lattice path $\LL^{\perp}_{\mathbf{a}}$ starts with an $E$ step and ends with an $N$ step, while $\LL_{\mathbf{b}}$ starts with an $N$ step and ends with an $E$ step.

Seeking contradiction, suppose that the two paths intersect at some point $(i,j)$ with $0 < i < p$ and $0 < j < q$ (these bounds follow from the observations in the previous paragraph), and let $(i,j)$ denote the last such point (starting from $(0,0)$). Since $\LL_{\mathbf{b}}$ lies weakly above $\LL^{\perp}_{\mathbf{a}}$, this implies that the next step in $\LL^{\perp}_{\mathbf{a}}$ must be an $E$ step, i.e. the $(i+1)$-th $E$ step in $\LL^{\perp}_{\mathbf{a}}$ has $y$-coordinate $j$. Now as previously, consider $p' := p^{v^p_1 +}$ with associated lattice paths $\LL^{\perp}_{\mathbf{a'}}$ and $\LL_{\mathbf{b'}}$. By Corollary~\ref{cor:GPPF_charac} and Proposition~\ref{pro:GPF_PF_Tri}, $p'$ is a $(p,q)$-parking function, so $\LL_{\mathbf{b'}}$ lies weakly above $\LL^{\perp}_{\mathbf{a'}}$. However, by preceding remarks, the $(i+1)$-th $E$ step in $\LL^{\perp}_{\mathbf{a'}}$ has $y$-coordinate $j+1$, while the $y$-coordinate of the $(i+1)$-th $E$ step in $\LL_{\mathbf{b'}}$ is the same as that of the $i$-th $E$ step in $\LL_{\mathbf{b}}$, and since $\LL_{\mathbf{b}}$ hits the point $(i,j)$, this value must be at most $j < j+1$. This gives the desired contradiction, and completes the proof that $\PPF{p,q} = \PGPF{K_{p,q}^0}$. The enumeration is given by~\cite[Theorem~4.9]{ABHKMMY}.
\end{proof}

\subsection{The complete bipartite case}\label{subsec:bipartite}

In this section, we consider the complete bipartite graph $K_{p*,q}$. This graph consists of two independent sets $P = \{v_0^p, v_1^p, \dots, v_p^p\}$ (including the sink $v_0^p$) and $Q = \{v_1^q, \dots, v_q^q\}$, and edges between each pair $(v_a^p, v_b^q) \in P \times Q$ (see Figure~\ref{fig:bipartite}). A parking function $p \in \GPF{K_{p*,q}}$ is called \emph{increasing} if $p\left(v_1^p \right) \leq \dots \leq p\left(v_p^p\right)$ and $p\left(v_1^q \right) \leq \dots \leq p\left(v_q^q\right)$. We denote by $\incGPF{K_{p*,q}}$, resp.\ $\incPGPF{K_{p*,q}}$ the set of increasing, resp.\ increasing prime, $K_{p*,q}$-parking functions.

\begin{figure}[ht]
 \centering
  \begin{tikzpicture}[scale=0.35]
 \node (0) [draw, rectangle, fill=black] at (1,0) {};
 \node [MyNode] (1) at (3,0) {};
 \node [MyNode] (2) at (5,0) {};
 \node [MyNode] (3) at (7,0) {};
 
 \node [MyNode] (4) at (3,-3) {};
 \node [MyNode] (5) at (5,-3) {};
     
 \foreach \xx in {0,1,2,3}
   \foreach \yy in {4,5}
     \draw [thick] (\xx)--(\yy);

 \node at (1,1.2) {$v_0^p$};
 \node at (3,1.2) {$v_1^p$};
 \node at (5,1.2) {$v_2^p$};
 \node at (7,1.2) {$v_3^p$};
 \node at (3,-4.2) {$v_1^q$};
 \node at (5,-4.2) {$v_2^q$};
  \end{tikzpicture}
  \caption{The complete bipartite graph $K_{3*,2}$.\label{fig:bipartite}}
\end{figure}
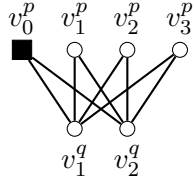

\begin{theorem}\label{thm:bipartite_bij}
There is a bijection from $\incPGPF{K_{p*,q}}$ to $ \incGPF{K_{(p-1)*,q}}$. In particular, we have $\left\vert \incPGPF{K_{p*,q}} \right\vert = \left\vert \incGPF{K_{(p-1)*,q}} \right\vert = \frac{1}{p+q-1} \binom{p+q-1}{p} \binom{p+q-1}{p-1}$.
\end{theorem}

\begin{proof}
We first claim that if $p \in \PGPF{K_{p*,q}}$ is a prime parking function, then there exists $v \in P$ such that $p(v) = 1$. For this, first choose some $u \in Q$ such that $p(u) = 1$ (such a vertex always exists by taking $S = \V = V \setminus \{v_0^p\}$ in Definition~\ref{def:GPF}). Now by definition, we have $u \in V_M(p)$ so that $p^{u+} \in \GPF{K_{p*,q}}$ by Corollary~\ref{cor:GPPF_charac}. By definition, we have $p^{u+} = p + \sum\limits_{w \in Q \setminus \{u\}} \1{w}$, and in particular $p^{u+}(w) > 1$ for all $w \in Q \setminus \{u\}$. We then take $S = V \setminus \{v_0^p, u\}$ in Definition~\ref{def:GPF}: there exists $v \in S$ such that $p^{u+}(v) \leq \dgr{\{v_0^p, u\}}{v}$. However, by preceding remarks we cannot have $v \in Q$, since otherwise we would have $1 < p^{u+}(v) \leq \dgr{\{v_0^p, u\}}{v} = 1$, a contradiction. Therefore there exists $v \in P$ such that $1 \leq p(v) = p^{u+}(v) \leq \dgr{\{v_0^p, u\}}{v} = 1$, and the claim is proved.

In particular, if $p \in \incPGPF{K_{p*,q}}$ is an increasing prime parking function, we have $p(v_1^p) = 1$. To simplify notation and avoid relabelling vertices, we consider the graph $K_{(p-1)*,q}$ to simply be the graph $K_{p*,q}$ with $v_1^p$ deleted. As in the proof of Lemma~\ref{lem:sym_PPFtoPF} we will write $V$, resp.\ $V'$, for the set of vertices of $K_{p*,q}$, resp.\ $K_{(p-1)*,q}$, using the $\tilde{\square}$ notation to denote non-sink vertices. We then claim that the deletion map $\Del{v_1^p}$ is a bijection from $\incPGPF{K_{p*,q}}$ to $ \incGPF{K_{(p-1)*,q}}$.

First, let us show that $\Del{v_1^p}\left(\incPGPF{K_{p*,q}}\right) \subseteq \incGPF{K_{(p-1)*,q}}$. Note that we cannot apply Lemma~\ref{lem:sym_PPFtoPF} in this case since $v_1^p \notin V_M(p)$ (we have $\mult{v_1^p}{v_0^p} = 0$).  
Define $p' := \Del{v_1^p}(p)$, and let $S \subseteq \V'$ be non-empty. We write $S = S_P \cup S_Q$ with $S_P \subseteq \{v_2^p, \dots, v_p^p \}$ and $S_Q \subseteq \{v_1^q, \dots, v_q^q\}$. 
We wish to show that there exists $v \in S$ such that $p'(v) \leq \dgr{V' \setminus S}{v}$. As previously, since $p$ is a parking function, we must have $p(v_1^q) = p'(v_1^q) = 1$ (using the fact that $p$ is increasing). In particular, since the sink vertex is always in $V' \setminus S$, if $S_Q$ contains $v_1^q$, we may simply choose $v = v_1^q$.

We therefore assume that $v_1^q \notin S_Q$. Now since $p$ is prime, we have $p^{v_1^q +} \in \incGPF{K_{p*,q}}$ by Corollary~\ref{cor:GPPF_charac}. 
Note that we have $p^{v_1^q +} = p + \sum \limits_{j=2}^{q} \1{v_j^q}$. Since $S \subseteq \V$, by Definition~\ref{def:GPF}, there exists $v \in S$ such that $p^{v_1^q +}(v) \leq \dgr{V \setminus S}{v}$. 
Now if $v \in S_P$, we have $p'(v) = p(v) = p^{v_1^q +}(v) \leq \dgr{V \setminus S}{v} = \dgr{V' \setminus S}{v}$, since $V' = V \setminus \{v_1^p\}$ but $v$ and $v_1^p$ do not share an edge when $v \in S_P$. 
Similarly, if $v \in S_Q$, we have $p'(v) = p(v) = p^{v_1^q +}(v) - 1 \leq \dgr{V \setminus S}{v} - 1 = \dgr{V' \setminus S}{v}$, since $v$ and $v_1^p$ share an edge in this case. In both cases we get $p'(v) \leq \dgr{V' \setminus S}{v}$, showing that $p'$ is indeed a $K_{(p-1)*,q}$-parking function, as desired.

As such, the map $\Del{v_1^p} : \incPGPF{K_{p*,q}} \rightarrow \incGPF{K_{(p-1)*,q}}, p \mapsto p'$ is well-defined, and by construction is clearly injective. 
We show that it is surjective, i.e. that inserting a vertex $v$ with $p(v)=1$ into the $P'$-component of a $K_{(p-1)*,q}$-parking function yields a prime $K_{p*,q}$-parking function. Let $p' \in \incGPF{K_{(p-1)*,q}}$ and $p$ be the increasing function on $\V\left(K_{p*,q}\right)$ such that $p' = \Del{v_1^p}(p)$. Clearly, since $p'$ is a $K_{(p-1)*,q}$-parking function, $p$ is a $K_{p*,q}$-parking function. To show that $p$ is prime, it is sufficient to show that $p^{v_1^q + }$ is a  $K_{p*,q}$-parking function (by symmetry and Corollary~\ref{cor:GPPF_charac}). Fix $S = S_P \cup S_Q \subseteq \V\left(K_{p*,q}\right)$. As above, we may assume that $v_1^q \notin S_Q$. As such, if $v_1^p \in S_P$, we have $p^{v_1^q +} \left( v_1^p \right) = p \left( v_1^p \right) = 1 = \mult{v_1^p}{v_1^q} \leq \dgr{V \setminus S}{v_1^p}$ as desired, so we may assume $v_1^p \notin S_P$. The proof then follows as above by exploiting the fact that $p'$ is a $K_{(p-1)*,q}$-parking function and considering $S$ as a subset of $\V\left(K_{(p-1)*,q}\right)$.

The enumerative result follows from~\cite[Corollary 3.10]{DLB} and the bijection from (increasing) $K_{(p-1)*,q}$-parking functions to recurrent configurations of the ASM on $K_{(p-1)*,q}$ (Theorem~\ref{thm:bij_rec_GPF} in this paper).
\end{proof}

\subsection{The complete split case}\label{subsec:split}

In this section, we consider the case where $G$ is the complete split graph $S_{m*,n}$. This graph consists of a clique with vertices $C = \left\{v_0^c, \dots, v_m^c \right\}$ (including the sink $v_0^c$), an independent set $I = \left\{ v_1^i, \dots, v_n^i \right\}$, and edges between each pair $(v_a^c, c_b^i) \in C \times I$ (see Figure~\ref{fig:split}). As in the complete bipartite case, we are interested in \emph{increasing} (prime) $S_{m*,n}$-parking functions, with analogous definitions and notation.

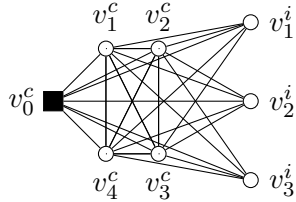
\begin{figure}[ht]
 \centering
  \begin{tikzpicture}[scale=0.35]
 \node (0) [draw, rectangle, fill=black] at (0,0) {};
 \node [MyNode] (1) at (2,2) {};
 \node [MyNode] (2) at (4,2) {};
 \node [MyNode] (3) at (2,-2) {};
 \node [MyNode] (4) at (4,-2) {};
 
 \node [MyNode] (5) at (7.5,-3) {};
 \node [MyNode] (6) at (7.5,0) {};
 \node [MyNode] (7) at (7.5,3) {};\

 \foreach \xx in {0,1,2,3,4}
   \foreach \yy in {1,...,7}
     \draw (\xx)--(\yy);

 \node at (-1.2,0) {$v_0^c$};
 \node at (2,3.2) {$v_1^c$};
 \node at (4,3.2) {$v_2^c$};
 \node at (2,-3.2) {$v_4^c$};
 \node at (4,-3.2) {$v_3^c$};
 \node at (8.7,3) {$v_1^i$};
 \node at (8.7,0) {$v_2^i$};
 \node at (8.7,-3) {$v_3^i$};
  \end{tikzpicture}
  \caption{The complete split graph $S_{4*,3}$.\label{fig:split}}
\end{figure}

\begin{theorem}\label{thm:split_bij}
There is a bijection from $\incPGPF{S_{m*,n}}$ to $ \incGPF{S_{(m-1)*,n}}$. In particular, we have $\left\vert \incPGPF{S_{m*,n}} \right\vert = \left\vert \incGPF{S_{(m-1)*,n}} \right\vert = \frac{1}{m} \binom{2m-2}{m-1} \binom{2m+n-2}{n}$.
\end{theorem}

\begin{proof}
The bijection and its proof are very similar to that of Theorem~\ref{thm:bipartite_bij}, so we allow ourselves to be briefer here. We first claim that if $p \in \PGPF{S_{m*,n}}$ is a prime parking function, then there exists $v \in C$ such that $p(v) = 1$. Indeed, since $p$ is a parking function and the graph $S_{m*,n}$ is simple, there exists $w \in V$ such that $p(w) = 1$ (choose any $w \in V_M(p)$). If $w \in C$ then we are done. Otherwise, if $w \in I$, we consider the parking function $p^{w+}$ and set $S = \V \setminus \{w\}$. Then there exists $v \in S$ such that $p^{w+}(v) \leq \dgr{S^c}{v} = \dgr{\{v_0^c, w\}}{v}$. As in the proof of Theorem~\ref{thm:bipartite_bij}, we cannot have $v \in I$ since otherwise $p^{w+}(v) = p(v) + 1 \geq 2 > \dgr{\{v_0^c, w\}}{v} = 1$. Therefore, we have $v \in C$ and $1 \leq p(v) = p^{w+}(v) - 1 \leq \dgr{\{v_0^c, w\}}{v} -1 = 2-1 = 1$, so $p(v) = 1$ as desired.

We then consider the map $\Del{v_1^c}: \incPGPF{S_{m*,n}} \rightarrow \incGPF{S_{(m-1)*,n}}$ which consists in deleting the vertex $v_1^c$ (satisfying $p\left(v_1^c\right) = 1$). The map is well-defined by Lemma~\ref{lem:sym_PPFtoPF} (here the assumptions are satisfied), and clearly injective. To show surjectivity, consider a parking function $p' \in \GPF{S_{(m-1)*,n}}$ and let $p$ be the function on $S_{m*,n}$ obtained by inserting the additional vertex $v_1^c$ with $p\left(v_1^c\right) = 1$. We wish to show that $p$ is a prime parking function. That $p$ is a parking function follows from the fact that $p'$ is, and as in Theorem~\ref{thm:bipartite_bij} we show primeness through Corollary~\ref{cor:GPPF_charac}.

Let $v \in V_M(p)$, i.e.\ $p(v) = 1$ (here all vertices $w$ satisfy $\mult{w}{s} = 1$). We wish to show that $p^{v+} = p + \func{1} - \1{v}$ is a parking function. Let $S \subseteq \V$ be a non-empty subset of non-sink vertices of $S_{m*,n}$. If $v \in S$, then we have $p^{v+}(v) = p(v) = 1 = \mult{v}{v_0^c} \leq \dgr{V \setminus S}{v}$, and there is nothing more to show. We therefore assume that $v \notin S$, i.e. that $S \subseteq \V'$ is a subset of non-sink vertices in $S_{(m-1)*,n}$. Then since $p'$ is a $S_{(m-1)*,n}$-parking function, there exists $w \in S$ such that $p'(w) \leq \dgr{V' \setminus S}{w}$, yielding $p^{v+}(w) = p(w) + 1 = p'(w) + 1 \leq \dgr{V' \setminus S}{w} + 1 = \dgr{V' \setminus S}{w} + \mult{w}{v_1^c} = \dgr{V \setminus S}{w}$, as desired. This completes the proof that $\Del{v_1^c}$ is the desired bijection.

The enumeration result then follows from~\cite[Corollary 3.4]{Duk} and the bijection between (increasing) $S_{(m-1)*,n}$-parking functions and (decreasing) recurrent configurations of the ASM on $S_{(m-1)*,n}$.
\end{proof}


\section{Conclusion and future work}\label{sec:conc}

The main goal of our paper was to introduce a notion of \emph{primeness} of $G$-parking functions for general graphs $G$. We define prime $G$-parking functions as parking functions that cannot be decomposed in a certain sense (Definition~\ref{def:prime_GPF}). Our notion addresses an open problem from~\cite{ABHKMMY} by generalising the concepts of primeness for classical parking functions (Proposition~\ref{pro:classical_PPF_complete_graph}) and for $(p,q)$-parking functions (Theorem~\ref{thm:tripartite_pqPF}).

We then studied this new concept of primeness under the well-known duality between $G$-parking functions and recurrent configurations for the Abelian sandpile model (ASM) on $G$ (see Theorem~\ref{thm:bij_rec_GPF} in this paper). We showed that prime $G$-parking functions correspond to so-caled \emph{strongly recurrent} configurations for the ASM on $G$ (Theorem~\ref{thm:bij_sr_GPPF}). These are recurrent configurations which remain recurrent after removing grains of sand from certain vertices (Definition~\ref{def:SR}).

Finally, we studied the sets of prime $G$-parking functions (equivalently, of strongly recurrent configurations for the ASM on $G$), in cases where $G$ belongs to graph families with high degrees of symmetry. The cases considered include wheel graphs (Proposition~\ref{pro:wheel_sr}), complete bipartite graphs (Theorem~\ref{thm:bipartite_bij}) and complete split graphs (Theorem~\ref{thm:split_bij}). Our results include both characterisations and enumerations of prime parking functions on these graph families.

One potentially interesting open problem is that of uniqueness of prime decompositions. In Remark~\ref{rem:decomp} we noted that prime decompositions of $G$-parking functions are not necessarily unique. However, they are known to be unique (up to symmetry) for classical parking functions (equivalently, when $G = K_n^0$) and for $(p,q)$-parking functions (equivalently, when $G = K_{p,q}^0$, see~\cite[Section~4.1]{ABHKMMY}). 

\begin{problem}\label{pb:unique_prime_decomp}
Under what conditions on the graph $G$ do all $G$-parking functions have a unique prime decomposition (up to the symmetries of $G$)?
\end{problem}

In terms of future research, Armon \emph{et al.}~\cite{ABHKMMY} noted of generalised prime parking functions that ``Much remains to be investigated in this area''. We echo their thoughts, in particular in terms of defining and studying concepts of primeness for some of the variations of parking functions which have been widely studied in the literature (see the survey in Section~\ref{subsec:PF_intro} of this paper). It would also be interesting to further the combinatorial study of prime $G$-parking functions on graph families that was instigated in this paper.

\section*{Acknowledgments}

The authors have no competing interests to declare that are relevant to the content of this article. The research leading to these results is partially supported by the Research Development Fund of Xi'an Jiaotong-Liverpool University, grant number RDF-22-01-089, and by the Postgraduate Research Scholarship of Xi'an Jiaotong-Liverpool University, grant number PGRS2012026.

\bibliographystyle{plain}
\bibliography{SR_primeGPF_bibliography}

\end{document}